\documentclass[11pt, english]{article}

\usepackage[margin= 2 cm,bottom=21mm, top= 20mm]{geometry}

\usepackage{thm-restate}
\usepackage[dvipsnames]{xcolor}
\usepackage[square,sort,comma,numbers]{natbib}
\usepackage{amsthm}
\usepackage{amsmath}
\usepackage{amssymb}
\usepackage{mathrsfs}
\usepackage{setspace}
\usepackage{mathtools}
\usepackage{graphicx}
\graphicspath{ {./images/} }
\usepackage[hidelinks]{hyperref}
\usepackage{cleveref}
\usepackage{hyperref}
\usepackage{enumitem}
\usepackage{framed}
\usepackage{subcaption}
\usepackage{xspace}
\usepackage{thmtools} 
\usepackage{thm-restate}
\usepackage{tikz-feynman}
\usepackage{thmtools}
\usepackage{thm-restate}
\usepackage{ifthen} 
\usepackage{floatrow}
\usepackage{comment}

\usepackage{tikz}
\usepackage{mathdots}
\usepackage{xcolor}
\usepackage{diagbox}
\usepackage{colortbl}
\usepackage[absolute,overlay]{textpos}

\usetikzlibrary{shapes.misc}
\usetikzlibrary{decorations.pathmorphing}

\tikzset{snake it/.style={decorate, decoration=snake}}
\usetikzlibrary{math}

\usetikzlibrary{calc}
\usetikzlibrary{decorations.pathreplacing}
\usetikzlibrary{positioning,patterns}
\usetikzlibrary{arrows,shapes,positioning}
\usetikzlibrary{decorations.markings}

\tikzstyle{edge}=[very thick]
\definecolor{bostonuniversityred}{rgb}{0.8, 0.0, 0.0}
\definecolor{arsenic}{rgb}{0.23, 0.27, 0.29}
\tikzstyle{diredge}=[postaction={decorate,decoration={markings,
		mark=at position .95 with {\arrow[scale = 1]{stealth};}}}]
\tikzset{
    arrow/.style={decoration={markings, mark=at position 0.7 with
    {\fill(-0.09*#1,-0.03*#1) -- (0,0) -- (-0.09*#1,0.03*#1) -- cycle;}}, postaction={decorate}},
    arrow/.default=1
}
\tikzset{
    arow/.style={decoration={markings, mark=at position 1 with
    {\fill(-0.09*#1,-0.03*#1) -- (0,0) -- (-0.09*#1,0.03*#1) -- cycle;}}, postaction={decorate}},
    arow/.default=1
}
\tikzset{
    arrrow/.style={decoration={markings, mark=at position 0.9 with
    {\fill(-0.09*#1,-0.03*#1) -- (0,0) -- (-0.09*#1,0.03*#1) -- cycle;}}, postaction={decorate}},
    arow/.default=1
}

\newcommand{\cB}{\ensuremath{\mathcal B}}

\newcommand{\cD}{\ensuremath{\mathcal D}}

\newcommand{\cF}{\ensuremath{\mathcal F}}

\newcommand{\cR}{\ensuremath{\mathcal R}}
\newcommand{\cS}{\ensuremath{\mathcal S}}

\newcommand{\fitellipsis}[2] 
{\draw [fill=white]let \p1=(#1), \p2=(#2), \n1={atan2(\y2-\y1,\x2-\x1)}, \n2={veclen(\y2-\y1,\x2-\x1)}
    in ($ (\p1)!0.5!(\p2) $) ellipse [ x radius=\n2/2+0cm, y radius=1.1cm, rotate=\n1];
}
\newcommand{\Fitellipsis}[2] 
{\draw [fill=white]let \p1=(#1), \p2=(#2), \n1={atan2(\y2-\y1,\x2-\x1)}, \n2={veclen(\y2-\y1,\x2-\x1)}
    in ($ (\p1)!0.5!(\p2) $) ellipse [ x radius=\n2/2+0cm, y radius=1.4cm, rotate=\n1];
}

\theoremstyle{plain}

\newtheorem*{thm*}{Theorem}
\newtheorem{thm}{Theorem}[section]
\Crefname{thm}{Theorem}{Theorems}

\newtheorem*{lem*}{Lemma}
\newtheorem{lem}[thm]{Lemma}
\Crefname{lem}{Lemma}{Lemmas}

\newtheorem*{claim*}{Claim}
\newtheorem{claim}[thm]{Claim}
\crefname{claim}{Claim}{Claims}
\Crefname{claim}{Claim}{Claims}

\newtheorem{prop}[thm]{Proposition}
\Crefname{prop}{Proposition}{Propositions}

\Crefname{remar}{Remark}{Remarks}

\newtheorem{cor}[thm]{Corollary}
\crefname{cor}{Corollary}{Corollaries}

\newtheorem*{conj*}{Conjecture}

\crefname{conj}{Conjecture}{Conjectures}

\Crefname{qn}{Question}{Questions}

\Crefname{obs}{Observation}{Observations}

\Crefname{ex}{Example}{Examples}

\theoremstyle{definition}

\Crefname{prob}{Problem}{Problems}

\newtheorem{defn}[thm]{Definition}
\Crefname{defn}{Definition}{Definitions}

\theoremstyle{remark}

\renewenvironment{proof}[1][]{\begin{trivlist}
\item[\hspace{\labelsep}{\bf\noindent Proof#1.\/}] }{\qed\end{trivlist}}

\newcommand{\remove}[1]{}

\newcommand{\eps}{\varepsilon}

\title{\vspace{-0.85 cm}
Size-Ramsey numbers of structurally sparse graphs}
\date{}

\author{
Nemanja Dragani\'c\thanks{
Department of Mathematics, ETH, Z\"urich, Switzerland. Research supported in part by SNSF grant 200021\_196965.
\newline
\emph{Emails}: \textbf{\{nemanja.draganic,david.munhacanascorreia\}@math.ethz.ch}.
}
\and
Marc Kaufmann\thanks{Institute of Theoretical Computer Science, Department of Computer Science, ETH, Z\"urich, Switzerland. M.K. was supported by SNSF grant 200021\_192079. K.P. was supported by grant no. CRSII5 173721 of the Swiss National Science Foundation. R.S. was supported by an ETH Z\"{u}rich Postdoctoral Fellowship.
\newline
\emph{Emails}: \textbf{\{marc.kaufmann,kalina.petrova,raphaelmario.steiner\}@inf.ethz.ch}.
}
\and
David Munh\'a Correia\footnotemark[1] 
\and
Kalina Petrova\footnotemark[2]
\and Raphael Steiner\footnotemark[2]
}
\begin{document} 
\maketitle
\begin{abstract}
Size-Ramsey numbers are a central notion in combinatorics and have been widely studied since their introduction by Erd\H{o}s, Faudree, Rousseau and Schelp in 1978. Research has mainly focused on the size-Ramsey numbers of $n$-vertex graphs with constant maximum degree $\Delta$. For example, graphs which also have constant treewidth are known to have linear size-Ramsey numbers. On the other extreme, the canonical examples of graphs of unbounded treewidth are the grid graphs, for which the best known bound has only very recently been improved from $O(n^{3/2})$ to $O(n^{5/4})$ by Conlon, Nenadov and Truji\'c. In this paper, we study a common generalization of these problems and establish new bounds on the size-Ramsey numbers in terms of treewidth (which may grow as a function of $n$). As a special case, this yields a bound of $\tilde{O}(n^{3/2 - 1/2\Delta})$ for proper minor-closed classes of graphs. In particular, this bound applies to planar graphs, addressing a question of Kamcev, Liebenau, Wood and Yepremyan. 
Our proof combines methods from structural graph theory and classic Ramsey-theoretic embedding techniques, taking advantage of the product structure exhibited by graphs with bounded treewidth.

\end{abstract}
\section{Introduction}
The famous theorem of Ramsey states that for every positive integer $n$, there exists a finite number $N$ such that every complete graph on $N$ vertices whose edges are colored in red or in blue contains a monochromatic clique of order $n$ \cite{ramsey1929logic}. Motivated by this, we say that a \emph{host} graph $G$ is $k$-Ramsey for a graph $H$ if any $k$-coloring of the edges of $G$ yields a monochromatic copy of $H$, and we write $G\xrightarrow{k}H$.

Given a graph $H$, its size-Ramsey number $\hat{r}^k(H)$ is defined as the smallest number of edges in a graph $G$ such that $G\xrightarrow{k}H$. This notion measures the minimality of the host graph $G$ in a precise manner, and was introduced in 1978 by Erd\H{o}s, Faudree, Rousseau and Schelp \cite{erdös1978size}. Since then, it has been extensively studied and in particular in recent years it has gained a lot of attention.

Intuitively, sparse graphs are a good candidate for having low size-Ramsey number, but already double stars --- trees containing only two non-leaf vertices, each of degree linear in the size of the graph --- have quadratic size-Ramsey number~\cite{erdös1978size}. Thus, a natural class of graphs to study in this context is the class of bounded-degree graphs. Among the first instances of this was the result of  Beck \cite{beck1983pathsI} who proved that for paths it holds that $\hat{r}(P_n)\le 900 n$, answering a $\$100$-dollar question of Erd\H{o}s.
Furthermore, Friedman and Pippenger \cite{Friedman1987ExpandingGC} proved that bounded-degree trees have linear size-Ramsey numbers (i.e., bounded from above by a constant times their order). More recently, it has been shown that graphs with bounded degree and bounded treewidth exhibit the same behaviour~\cite{kamcev2021bounded,berger2019size}, as well as \emph{long} subdivisions of bounded-degree graphs~\cite{draganic2022rolling}.\par

Perhaps surprisingly, not all bounded-degree graphs have this behaviour. In fact, there exist graphs with maximum degree three which serve as counterexamples, as shown by R{\"o}dl and Szemer\'edi \cite{rödl2000bounded}. More concretely, they constructed cubic graphs $H$ on $n$ vertices fulfilling $\hat{r}(H) \ge c \cdot n \cdot \log^{\frac{1}{60}}(n) $ for a universal constant $c>0$. Their lower bound has been improved very recently to $\hat{r}(H) \ge c \cdot n \cdot e^{c \sqrt{\log(n)}} $  by Tikhomirov \cite{tikhomirov2022large}, modifying their construction by a clever randomization trick. In spite of these advances, the conjecture of R{\"o}dl and Szemer\'edi that there exist bounded-degree $n$-vertex graphs $H$ with $\hat{r}(H)\ge n^{1+\varepsilon}$ for some absolute constant $\varepsilon>0$ remains open. 

Turning to general upper bounds, Kohayakawa, Schacht, R{\"o}dl and Szemer\'edi~\cite{kohayakawa2011sparse} showed that for graphs $H$ which have maximum degree $\Delta$ it holds that $\hat{r}(H)= O(n^{2-1/\Delta}\log^{1/\Delta}n)$, leaving a wide gap between the best known upper and lower bounds. Recently, Allen and B\"ottcher~\cite{allen2022partition} announced an improvement of this bound to $\hat{r}(H) = O(n^{2 -1/(\Delta-1)+o(1)})$ for $\Delta \geq 4$. For the case $\Delta=3$, the first and fourth authors~\cite{draganic2022size} showed that $\hat{r}(H)= O(n^{3/2+o(1)})$.
As mentioned above, the class of bounded-degree graphs with constant treewidth is far better understood, and in fact it is known that the size-Ramsey numbers of these graphs are linear in $n$~\cite{berger2019size, kamcev2021bounded}. Unfortunately, the proofs in both \cite{berger2019size} and \cite{kamcev2021bounded} do not generalize beyond the setting of constant treewidth. In particular, if one allows the treewidth of the considered graphs $H$ to grow with $n$, then thus far no better upper bounds on $\hat{r}(H)$ than the bound of Kohayakawa et al.~(which only takes into account degree bounds, and not the structure of the graphs) were known~\cite{kohayakawa2011sparse} .  In this paper, we generalize the above-mentioned results to graph classes of unbounded treewidth. In particular, leveraging the treewidth gives a substantial improvement over the bound of Kohayakawa et al.

\begin{restatable}{thm}{Treewidth-ramsey}\label{thm:main-treewidth}
 Let $k\in \mathbb N$ and let $H$ be an $n$-vertex graph of constant maximum degree $\Delta$ and treewidth $t=t(n)$. Then 
 \begin{itemize}
     \item $\hat{r}_k(H)=O(nt\log n)$
     \item  Furthermore, if $t=\Omega(e^{\sqrt{\log n}})$ then we have\footnote{We use $\tilde O(f(n))$ to denote all functions which are in $O(f(n)(\log n)^C)$ for some absolute constant $C>0$.} $\hat{r}_k(H)=\tilde O(nt^{1-1/\Delta})$.
 \end{itemize}
\end{restatable}

 Treewidth is a key parameter in structural graph theory and has served as the watershed in Robertson and Seymour's seminal sequence of  papers on graph minors, culminating in what is now known as the Graph Minor Theorem. Tree decompositions, which are inherently related to treewidth, have enabled advances in algorithmic graph theory, with many NP-hard problems proving solvable in polynomial time for entire graph classes whose treewidth is known to be bounded.

 In fact, Theorem~\ref{thm:main-treewidth} is a special case of more general results, namely \Cref{lem:smallblowupsizeramsey} and \Cref{lem:blowupsizeramsey}, which we prove in Section \ref{sec:proof}. Indeed, there we show that any bounded-degree graph contained in a product of a bounded-degree tree and a clique of certain size has small size-Ramsey number. Since (as we will prove in Section~\ref{sec:planar}) graphs of bounded degree and treewidth $t$ are contained in such a product with certain parameters, Theorem~\ref{thm:main-treewidth} will follow.
 


As the treewidth of many graph classes is well understood, Theorem~\ref{thm:main-treewidth} may provide good upper bounds for the size-Ramsey numbers of graphs in those classes. One such example are proper minor-closed classes of graphs, for which it is known~\cite{alon1990balancedsep,dvorak2019balancedsep} that they have treewidth at most $O(\sqrt{n})$, and hence we have the following consequence of Theorem~\ref{thm:main-treewidth}.

\begin{thm}\label{thm:main}
Let $k \in \mathbb{N}$ and let $\mathcal{G}$ be a proper minor-closed class. Then for every $n$-vertex graph $G\in \mathcal{G}$ of maximum degree $\Delta$, we have $$\hat{r}_k(G)=\tilde O(n^{3/2-\frac{1}{2\Delta}}).$$
\end{thm}

A natural and well studied class of minor-closed graphs are planar graphs, and hence the bound above also applies to them. This addresses a question raised by Kam\v{c}ev, Liebenau, Wood and Yepremyan~\cite{kamcev2021bounded}, who asked for general bounds for bounded-degree planar graphs. Until recently, the best known bound for the two-dimensional grid graph on $n$-vertices was $O(n^{3/2})$, which was improved to $O(n^{5/4})$ by Conlon, Nenadov and Truji\'c~\cite{conlon2022size}. The only other bounded-degree planar graphs which were studied are cycles and bounded-degree trees, which have linear size-Ramsey numbers.

Let us also remark that all our results are of universality type. Namely, in our proofs we construct a host graph such that we can find all the graphs in the relevant class within the same color class. Results of this type were also obtained for size-Ramsey numbers of graphs in other classes \cite{Friedman1987ExpandingGC, draganic2021size,draganic2022rolling, kohayakawa2019size, kohayakawa2011sparse,lee2017ramsey}.

\textbf{Organization.}
        In Section~\ref{sec:planar} we prove several results which embed a graph with a given bound on its treewidth into the product of a bounded-degree tree and a clique. The results in this section might be of independent interest. In Section~\ref{sec:proof} we prove bounds on the size-Ramsey numbers of bounded-degree subgraphs of products of bounded-degree trees and cliques, which, combined with the results from Section~\ref{sec:planar} then allow us to prove our main results. We decided to defer some of the technical details of the proofs from Section~\ref{sec:proof} to the appendix (Section~\ref{sec:appendix}) to ease the readability and intuitive understanding of our proofs. 

\textbf{Notation.}
	For simplicity, we employ the following conventions. We omit rounding of real numbers to nearest integers whenever it is not of vital importance. For two constants $a,b$, we use $a \ll b$ to indicate that $b$ is sufficiently large as a function of $a$ for our proofs to work out. For example, we often use inequality chains like $a \gg b \gg c \gg d$, which also implies that in particular $a \gg bcd$. We make use of the \emph{strong product} operator, which is defined as follows. The strong product $G \boxtimes H$ of graphs $G$ and $H$ is the graph with vertex set $V(G) \times V(H) $ such that for $v_1, v_2 \in V(G)$ and $u_1,u_2 \in V(H)$, we have that $(v_1, u_1)$ is adjacent to $(v_2, u_2)$ in $G \boxtimes H$ if $v_1 = v_2$ and $\{u_1,u_2\} \in E(H)$, or $\{v_1,v_2\} \in E(G)$ and $u_1 = u_2$, or $\{v_1,v_2\} \in E(G)$ and $\{u_1,u_2\} \in E(H)$.
\section{Structural results}\label{sec:planar}
In this section, we prove the announced structural lemma for graphs of bounded treewidth. It says that every bounded-degree $n$-vertex graph $G$ of treewidth at most $t$ can be embedded as a subgraph into the strong product of a bounded-degree tree of size $O(n/(t\log n))$ and a clique of size $O(t\log n)$. In the case that $G$ belongs to a structurally sparse class of graphs, such as planar graphs or a proper minor-closed class, we can get rid of the $\log n$ factor in the above estimate, which can be used to further improve Theorem~\ref{thm:main} by removing additional hidden logarithmic factors. Evidently, the significance of this result does not lie in this minor improvement, but rather in the fact that it is tight (consider, say, the grid graph).

\begin{lem}\label{lem:embedintoproduct}
 Let $\Delta >0$. There exists $d>0$ depending solely on $\Delta$ such that the following hold.
\begin{enumerate}
\item[(i)]\label{i:first} For every $n$-vertex graph $G$ of maximum degree at most $\Delta$ and treewidth at most $t$ there exists a tree $T$ such that $G \subseteq T \boxtimes K_s$, where $n \geq s=\Theta_\Delta(t \log n)$, $v(T) = O_{\Delta}\big(\frac{n}{s}\big)$ and $\Delta(T)\le d$. 
\item[(ii)] Let $\mathcal{G}$ be a class of graphs closed under taking subgraphs, and let $t:\mathbb{R}_+\rightarrow \mathbb{R}_+$ be an increasing function in $\Omega(\log x)$ such that every $n$-vertex graph $G \in \mathcal{G}$ has treewidth at most $t(n)$. Suppose further that there exists a constant $\alpha>0$ such that $t(\lambda x)\le \lambda^\alpha t(x)$ holds for every $\lambda,x>0$. Then for every $n$-vertex graph $G \in \mathcal{G}$ of maximum degree at most $\Delta$ there is a tree $T$ such that $G \subseteq T \boxtimes K_s$, where $n \geq s=\Theta_{\Delta,\alpha}(t(n))$, $v(T) = O_{\Delta}\big(\frac{n}{s}\big)$ and $\Delta(T)\le d$. 
\end{enumerate}
\end{lem}

We remark that a very similar embedding lemma, without the $\log n$ factor, has been proved before by Ding and Oporowski~\cite{ding1995treedecomp} and later on improved by Wood~\cite{wood2009treedecomp}. These results state that every graph $G$ of maximum-degree $\Delta$ and treewidth at most $t$ is contained in $T \boxtimes K_s$, where $T$ is some tree and $s= O(\Delta t)$. However, for our purposes, it is very important to be able to carefully control both the size and the maximum degree of the tree $T$. In particular, we need $T$ to have constant degree, and this is not the case in the results in~\cite{ding1995treedecomp, wood2009treedecomp}. Motivated by Lemma~\ref{lem:embedintoproduct}, very recently Distel and Wood~\cite{distel2022tree} showed a strengthening of it by removing the logarithmic factor in  the first part of Lemma~\ref{lem:embedintoproduct}. 


Before moving on to the proof of Lemma~\ref{lem:embedintoproduct}, let us note the following immediate consequence of it. 

\begin{cor}\label{cor:planarandminorclosed}
    Let $\Delta>0$ and let $\mathcal{G}$ be a proper minor-closed class. There exist constants $c,d>0$ such that every $n$-vertex graph $G \in \mathcal{G}$ of maximum degree at most $\Delta$ is a subgraph of $T \boxtimes K_s$, where $s \le c\sqrt{n}$ and $T$ is a tree with $v(T) \le c\sqrt{n}$ and $\Delta(T) \le d$.
\end{cor}
\begin{proof}
    It is well-known~\cite{alon1990balancedsep,dvorak2019balancedsep} that for every proper minor-closed class $\mathcal{G}$ of graphs, there is a constant $C>0$ such that every $n$-vertex graph $G\in \mathcal{G}$ has treewidth at most $C\sqrt{n}$. The statement therefore immediately follows from Lemma~\ref{lem:embedintoproduct}~(ii) by setting $t(x)=C\sqrt{x}$ and $\alpha=\frac{1}{2}$. 
\end{proof}

We now move on to the proof of Lemma~\ref{lem:embedintoproduct}, which follows closely the arguments in~\cite{chung1990separator,chung1989universal}. A crucial tool to embed a given graph $G$ of treewidth $t$ into the strong product of a tree and a clique is a repeated application of the following well-known result from structural graph theory concerning the existence of so-called \emph{balanced separators} in graphs of bounded treewidth. 
\begin{lem}[Robertson and Seymour, cf.~statement (2.6) in~\cite{robertson1986balancedsep}]\label{lem:balancedseps}
Let $G$ be a graph of treewidth at most $t$ and order $n$. Then there exists a set $S\subseteq V(G)$ of size $|S|\le t+1$ and a partition of $V(G)\setminus S$ into sets $A$ and $B$ such that no edge in $G-S$ connects a vertex in $A$ to a vertex in $B$ and such that $|A|, |B| \le \frac{2}{3}n$.
\end{lem}

To prove Lemma~\ref{lem:embedintoproduct}, we first establish an extension of Lemma~\ref{lem:balancedseps} which works when the graph $G$ comes with a given coloring of its vertices with $k$ colors, and we wish to find a small separator $S$ in the graph such that the two sides $A$ and $B$ separated by it contain each at most half the vertices in each color. To find such separators, we make use of the celebrated necklace-splitting theorem due to Goldberg and West in the following form. 
\begin{thm}[Goldberg and West~\cite{goldberg1985bisection}]\label{thm:neclacesplitting}
For some $R \subseteq \{1,\ldots,n\}$, let $c:R \rightarrow \{1,\ldots,k\}$ be a (partial) $k$-coloring of the first $n$ integers. Then $\{1,\ldots,n\}$ 
 can be partitioned into $k+1$ pairwise disjoint intervals $I_1, \ldots,I_{k+1}$ such that for some $0 \le r \le k+1$ it holds that (denoting $X:=\bigcup_{i=1}^{r}{I_i}, Y:=\bigcup_{i=r+1}^{k+1}{I_i}$) $$\max\{|X \cap c^{-1}(i)|,|Y \cap c^{-1}(i)|\}\le \left\lceil\frac{|c^{-1}(i)|}{2}\right\rceil$$ for $i=1,\ldots,k$. 
\end{thm}

\begin{lem}\label{lem:ordering}
Let $G$ be a graph and let $t:\mathbb{R}_+\rightarrow \mathbb{R}_+$ be an increasing function such that every subgraph of $G$ with $x$ vertices has treewidth at most $t(x)$. 

Then there exists a linear ordering $v_1,\ldots,v_n$ of the vertices in $G$ and for each $i \in \{1,\ldots,n\}$ a set $S(v_i) \subseteq V(G)$ including $v_i$ such that
$$|S(v_i)|\le \sum_{j=0}^{\lceil \log_{3/2}(n) \rceil}{\left(t\left(\left(\frac{2}{3}\right)^j n\right)+1\right)}$$ and such that there are no edges between $\{v_1,\ldots,v_{i-1}\}\setminus S(v_i)$ and $\{v_{i+1},\ldots,v_n\}\setminus S(v_i)$ in $G$.  
\end{lem}
\begin{proof}
    We prove the statement by induction on $n$. The claim is trivially true for $n=1$ by setting $S(v_1):=\{v_1\}$, so moving on suppose that $n \ge 2$ and the claim is true for all subgraphs of $G$ on fewer vertices. Apply Lemma~\ref{lem:balancedseps} to find a set $S \subseteq V(G)$ of size at most $t(n)+1$ and a partition of $V(G)\setminus S$ into parts $A$ and $B$ with no edges in $G$ between them, such that $|A|, |B|\le \frac{2}{3}n$. By the inductive assumption, there exist linear orderings $v_1,\ldots,v_{|A|}$ and $w_1,\ldots,w_{|B|}$ of $A$ and $B$ and for each $i \in \{1,\ldots,|A|\}$ and $j \in \{1,\ldots,|B|\}$ sets $S_A(v_i)$ and $S_B(w_j)$ such that the following hold.  $S_A(v_i) \ni v_i$ separates $\{v_k|k<i\}$ and $\{v_k|k>i\}$ in $G[A]$, and $S_B(w_j) \ni w_j$ separates $\{w_k|k<j\}$ and $\{w_k|k>j\}$ in $G[B]$, and with $m:=\max\{|A|,|B|\} \le \frac{2}{3}n$ we have
    $$|S_A(v_i)|, |S_B(v_j)|\le \sum_{k=0}^{\lceil \log_{3/2}(m) \rceil}{\left(t\left(\left(\frac{2}{3}\right)^k m\right)+1\right)} $$
    $$\le \sum_{k=1}^{\lceil \log_{3/2}(m)+1 \rceil}{\left(t\left(\left(\frac{2}{3}\right)^k n\right)+1\right)} \le  \sum_{k=1}^{\lceil \log_{3/2}(n) \rceil}{\left(t\left(\left(\frac{2}{3}\right)^k n\right)+1\right)}.$$
    Let $s_1,\ldots,s_{|S|}$ be an arbitrary linear ordering of $S$, and consider the following linear order on $V(G)$: $v_1,\ldots,v_{|A|},s_1,\ldots,s_{|S|},w_1,\ldots,w_{|B|}$. Furthermore let $S(v_i):=S_A(v_i) \cup S$ for $i=1,\ldots,|A|$, $S(s_k):=S$ for $k=1,\ldots,|S|$, and $S(w_j):=S_B(w_j) \cup S$ for $j=1,\ldots,|B|$. It is easily verified that each of these sets separates their predecessors and successors in the ordering in $G$ and since $|S|\le t(n)+1=t((2/3)^0n)+1$, each of these sets is of size at most $\sum_{h=0}^{\lceil \log_{3/2}(n) \rceil}{\left(t\left(\left(\frac{2}{3}\right)^h n\right)+1\right)}$. This concludes the proof.
\end{proof}

Using the previous statement and Lemma~\ref{thm:neclacesplitting}, we immediately arrive at the following corollary.
\begin{cor}\label{cor:coloredseparation}
    Let $G$ be a graph given with a (partial) coloring $c:R\rightarrow \{1,\ldots,k\}$ of its vertices for some $R\subseteq V(G)$ and let $t:\mathbb{R}_+\rightarrow \mathbb{R}_+$ be an increasing function such that every subgraph of $G$ with $x$ vertices has treewidth at most $t(x)$.  
    Then there is a set $S \subseteq V(G)$ and a partition $A, B$ of $V(G) \setminus S$ such that 
    $$|S| \le k+k \sum_{i=0}^{\lceil \log_{3/2}(n) \rceil}{\left(t\left(\left(\frac{2}{3}\right)^in\right)+1\right)},$$ there are no connections between $A$ and $B$ in $G$, and $$\max\{|A\cap c^{-1}(i)|, |B\cap c^{-1}(i)|\} \le \frac{|c^{-1}(i)|}{2}$$ for each $i \in \{1,\ldots,k\}$. 
\end{cor}
\begin{proof}
    We apply  Lemma~\ref{lem:ordering} to $G$ to get the linear ordering $v_1,\ldots,v_n$ of $V(G)$. Next, we apply Theorem~\ref{thm:neclacesplitting} to the coloring $c(i):=c(v_i)$. Let $I_1, \ldots,I_r, I_{r+1},\ldots,I_k,I_{k+1}$ be an interval partition of $\{1,\ldots,n\}$ as guaranteed by the lemma. Let $x_1,\ldots,x_k$ be a list of $k$ elements in $\{1,\ldots,n\}$ such that any two distinct intervals in $I_1,I_2,\ldots,I_{k+1}$ are separated by at least one of them. We now define $S:=S(x_1) \cup S(x_2) \cup \cdots \cup S(x_k)$, which by Lemma~\ref{lem:ordering} is of size at most $k \sum_{i=0}^{\lceil \log_{3/2}(n) \rceil}{\left(t\left(\left(\frac{2}{3}\right)^in\right)+1\right)}$, and let $A:=\{v_i|i \in I_1 \cup \cdots \cup I_r\}\setminus S, B:=\{v_i|i \in I_{r+1}\cup \cdots \cup I_{k+1}\}\setminus S$. It now follows directly from Lemma~\ref{lem:ordering} that there are no edges in $G$ connecting $A$ and $B$ and directly from the properties guaranteed by Theorem~\ref{thm:neclacesplitting} that $\max\{|A\cap c^{-1}(i)|, |B\cap c^{-1}(i)|\} \le \left\lceil\frac{|c^{-1}(i)|}{2}\right\rceil$ for each $i \in \{1,\ldots,k\}$. In order to reduce from the above bound $\left\lceil \frac{|c^{-1}(i)|}{2}\right\rceil$ to the desired $\frac{|c^{-1}(i)|}{2}$ for each $i$, we may simply move for each color $i$ a vertex of that color from either $A$ or $B$ into $S$, thereby increasing the size of $S$ by at most $k$, and the statement follows.
\end{proof}

Using Corollary~\ref{cor:coloredseparation} repeatedly, we can now prove the following statement, which directly implies Lemma~\ref{lem:embedintoproduct}. 
\begin{lem}\label{lem:poweroftree}
Let $G$ be a graph of maximum degree at most $\Delta$ and let $t:\mathbb{R}^+\rightarrow \mathbb{R}^+$ be an increasing function such that every $x$-vertex subgraph of $G$ has treewidth at most $t(x)$. Let $$k:=\lceil \log_2(\Delta)\rceil+2,\,  s:=\left\lfloor k+k\sum_{i=0}^{\lceil \log_{3/2}(n) \rceil}{\left(t\left(\left(\frac{2}{3}\right)^in\right)+1\right)}\right\rfloor.$$ Then there exists a binary tree $T$ and a partition $(X_v)_{v \in V(T)}$ of $V(G)$ such that $|X_v|=2s$ for every non-leaf $v\in V(T)$, $|X_v|\le 2s$ for every leaf $v\in V(T)$, and such that any two adjacent vertices in $G$ are contained in sets $X_u, X_v$ for vertices $u, v \in V(T)$ at distance at most $k$ in $T$.
\end{lem}
\begin{proof}
To enable induction, we prove a slight strengthening of the statement in the lemma as follows:

\paragraph{Claim.}\emph{Given any sequence $C_1,\ldots,C_{k}$ of disjoint subsets of $V(G)$ such that $|C_i|\le 2^{i-1}\cdot s$ for each $i$, we can find a tree $T$ and a partition $(X_v)_{v\in V(T)}$ as in the lemma with the following additional property: For each $i=1,\ldots,k$ and every vertex $v$ in $T$ at distance at least $i$ from the root, we have $X_v \cap C_i=\varnothing$.}

\medskip

Let us now prove this claim by induction on $v(G)$. It trivially holds true if $v(G)\le 2s$, so suppose that $v(G)>2s$ and the claim is satisfied by every subgraph of $G$ with fewer than $v(G)$ vertices. Interpret the disjoint sets $C_1,\ldots,C_k$ as the color classes of a partial coloring of $V(G)$, and apply Corollary~\ref{cor:coloredseparation} to find a set $S$ of size at most $s$ in $G$ and a partition $(A,B)$ of $V(G)\setminus S$ such that $|C_i \cap A|, |C_i \cap B|\le \frac{|C_i|}{2}\le 2^{i-2}\cdot s$ for each $i$. Pick a set $X$ of size exactly $2s$ in $V(G)$ such that $C_1 \cup S \subseteq X$ (this is possible since $|C_1|, |S| \le s$ and $v(G)>2s$). Now, let $A':=A\setminus X$, $B':=B \setminus X$. For each $i=1,\ldots,k-1$, define $C_i^A:=C_{i+1}\cap A'$, $C_i^B:=C_{i+1}\cap B'$. Further, let $C_k^A:=(N_G(X) \cap A')\setminus \bigcup_{i=1}^{k-1}{C_i^A}$ and $C_k^B:=(N_G(X) \cap B')\setminus \bigcup_{i=1}^{k-1}{C_i^B}$. From the above we have $|C_i^A|, |C_i^B| \le 2^{(i+1)-2}s=2^{i-1}s$ for $i=1,\ldots,k-1$, as well as $|C_k^A|, |C_k^B| \le |N_G(X)|\le \Delta\cdot 2s \le 2^{k-1}s$ since $G$ is of maximum degree at most $\Delta$ and by our choice of $k$.

Therefore, we can apply induction to $G[A']$ with the sets $C_1^A,\ldots,C_k^A$ and to $G[B']$ with the sets $C_1^B,\ldots,C_k^B$ to find binary trees $T_A$ and $T_B$, and partitions $(X_v)_{v\in V(T_A)}, (X_v)_{v \in V(T_B)}$ of $A'$ and $B'$ satisfying the inductive claim. We now form a binary tree $T$ by adding a new root vertex $r$ and connecting it to the roots of the binary trees $T_A$ and $T_B$, and setting $X_r:=X$. We claim that this tree and the partition satisfy the claim. Let us first verify the condition with respect to the sets $C_1,\ldots,C_k$. For $i=1$, and any vertex $v$ at distance at least $1$ from $r$ in $T$, we have that $X_v$ is disjoint from $X\supseteq C_1$. Further, for $i \in \{2,\ldots,k\}$, every vertex $v$ at distance at least $i$ from $r$ in $T$ has distance at least $i-1$ from the root in either $T_A$ or $T_B$, thus implying that $X_v$ is disjoint from $C_{i-1}^A=C_i \cap A'$ respectively $C_{i-1}^B=C_i \cap B'$, and therefore $X_v \cap C_i=\varnothing$ in each case. 

Hence, all that remains to be verified is that any two adjacent vertices in $G$ are contained in sets $X_u, X_v$ for vertices $u,v$ at distance at most $k$ in $T$. So consider any two vertices $u, v$ in $T$ such that there exists an edge in $G$ with endpoints in $X_u$ and $X_v$. If $u, v \neq r$ then either both of $u,v$ are contained in $T_A$ or both in $T_B$, and then it follows from the validity of the inductive claim for $T_A$ and $T_B$ that $u$ and $v$ are at distance at most $k$ in $T$. Next consider the case $u=r$, and w.l.o.g. assume $v \in V(T_A)$. Then $\varnothing \neq X_v \cap N_G(X_u)=X_v \cap N_G(X)=X_v \cap (N_G(X)\cap A')$, so in particular $X_v$ intersects $C_k^A$ or at least one of $C_1^A,\ldots,C_{k-1}^A$, let's say $C_i^A \cap X_v \neq \varnothing$ for some $i\in\{1,\ldots,k\}$. But then by our inductive call to $G[A']$, $v$ has to be at distance at most $i-1$ from the root in $T_A$. But this means that $v$ is at distance at most $i$ from the root $r$ in $T$, showing that $u=r$ and $v$ are at distance at most $i\le k$ also in this last case. Thus, $G$ satisfies the inductive assertion with respect to $C_1,\ldots,C_k$, concluding the proof of the lemma.
\end{proof}
We are now ready to show Lemma~\ref{lem:embedintoproduct}.
\begin{proof}[~of Lemma~\ref{lem:embedintoproduct}]
We do the proof of both cases (i) and (ii) of the lemma at once. Every subgraph of $G$ on $x$ vertices has treewidth at most $t(x)$, where we let $t(x):=t$ for each $x$ if we are in case (i) of the lemma. 
Let $k$ and $s$ be defined as in Lemma~\ref{lem:poweroftree} and note that we obtain $s=\Theta_\Delta(t\log n)$ in case (i) and 
$$k(t(n)+1)<s \le k+k\sum_{i=0}^{\lceil\log_{3/2}(n)\rceil}{\left(\left(\frac{2}{3}\right)^{\alpha i}t(n)+1\right)}$$ $$\le k+\frac{k}{1-(2/3)^\alpha}t(n)+k(\lceil\log_{3/2}(n)\rceil+1)=O_{\Delta,\alpha}(t(n))+O_{\Delta}(\log n)=O_{\Delta,\alpha}(t(n)),$$ hence $s=\Theta_{\Delta,\alpha}(t(n))$ in case (ii).

We can now apply Lemma~\ref{lem:poweroftree} to find a binary tree $T$ and a partition $(X_v)_{v\in T}$ of the vertex set of $G$ such that $|X_v|=2s$ for every non-leaf $v$ of $T$, $|X_v| \le 2s$ for every leaf $v$ of $T$, and such that whenever $X_u$ and $X_v$ have a connecting edge in $G$, then $u$ and $v$ are at distance at most $k=\lceil\log_2(\Delta)\rceil+2$ in $T$. Note that the number of non-leaves in $T$ is at most $\frac{n}{2s}$, and since a binary tree of size at least $3$ has at most one more leaf than non-leaf, we find that $v(T)\le \frac{n}{s}+1$. 
Furthermore, we can see 
from the above that $G$ is a subgraph of $T^k \boxtimes K_{2s}$, where $T^k$ is the graph on the same vertex set as $T$ in which two vertices are adjacent iff they are at distance at most $k$ from each other in $T$. Notice that since $T$ is a binary tree, it is not hard to show that $T^k\subseteq T' \boxtimes K_{2^k}$ where $T'$ is a tree such that $v(T') \le v(T)\le \frac{n}{s}+1$ and $\Delta(T')\le 1+2^{k}=:d$. All in all, we find that $G \subseteq (T' \boxtimes K_{2^k}) \boxtimes K_{2s} \simeq T' \boxtimes K_{2^{k+1}s}$. Since $2^{k+1}s=\Theta_{\Delta}(s)$, the assertion of the lemma in both cases now follows from the above asymptotic estimates of $s$. 
\end{proof}

\section{Proof of main result}\label{sec:proof}
In this section we introduce some tools that we need, and give the proofs of our main results. We start with the main result from~\cite{berger2019size}, which states that graphs with constant treewidth and constant degree have linear size-Ramsey numbers, and moreover, that there exists a host graph for this class of graphs which also has constant maximum degree.
\begin{thm}[Corollary 2 in \cite{berger2019size}]
     \label{cor:host-graph-bounded-treewidth}
    For any positive integers $k, \Delta, t$, and for every $n$ large enough, there exists a graph $G$ with $O(n)$ vertices and constant maximum degree such that for every $n$-vertex graph $H$ of maximum degree $\Delta$ and treewidth at most $t$, it holds that $G \rightarrow_k H$.
\end{thm}
The theorem above is stated in~\cite{berger2019size} without the maximum degree condition on $G$, but the construction in~\cite{berger2019size} clearly satisfies it. 

\subsection{Regularity method}
Similarly to previous work on size-Ramsey numbers, to find monochromatic subgraphs in our host graph, we make use of the celebrated \emph{sparse regularity method} (for a detailed exposition of which we refer the reader to~\cite{gerke2005sparse}). Instead of working with regular pairs as usual, it is enough for our purposes to only have lower bounds on the density of subgraphs, as expressed in the following definition.
\begin{defn}
\label{def:densepair}
Let $\alpha, \eps >0$ and $0 < p \leq 1$. For a graph $G=(V,E)$ and disjoint sets $X,Y \subseteq V$, we say that $(X,Y)$ is \emph{$(\eps, \alpha, p)$-dense} if for any $X' \subseteq X, Y' \subseteq Y$ with $|X'| \geq \eps|X|$ and $|Y'| \geq \eps |Y|$, we have
$$ d_G(X',Y') \geq (\alpha - \eps)p.$$
\end{defn}

To be able to apply the sparse regularity method, we need our host graph to have somewhat `uniformly' distributed edges, typically captured by the concept of $(\lambda,p)$-uniformity as follows.
\begin{defn}
    A graph $G=(V,E)$ is $(\lambda, p)$-uniform if for all disjoint $U,W \subseteq V$ with $|U|,|W| \geq \lambda |V|$, we have $(1-\lambda) p\le d_G(U,W) \le (1+\lambda)p$. If $G$ is bipartite with bipartition $V = A \cup B$, then we say that $G$ is $(\lambda, p)$-uniform if the same holds for all $U \subseteq A$ and $W \subseteq B$ of size at least $\lambda |A|$ and $\lambda |B|$ respectively.
\end{defn}

The next lemma, which we will use in the proofs of Lemmas \ref{lem:smallblowupsizeramsey} and \ref{lem:blowupsizeramsey}, allows us to find certain useful structure in colored blow-ups of bounded-degree graphs. Namely, it turns out one can choose a small fraction of the vertices in each part of the blow-up in such a way that the bipartite graph between each pair of these subparts has at least one color with the lower-regularity properties captured by Definition~\ref{def:densepair} and some minimum density.
\begin{lem}[Lemma 2.3 in~\cite{conlon2022size}]
\label{lem:get-regularity}
For every $k,\Delta \geq 2$ and $\eps > 0$, there exists $\lambda > 0$ such that the following holds for every $p \in (0,1]$. Let $H$ be a graph on at least two vertices with $\Delta(H) \leq \Delta$ and let $\Gamma$ be obtained by replacing every $x \in V(H)$ with an independent set $V_x$ of sufficiently large order $n$ and every $xy \in E(H)$ by a $(\lambda,p)$-uniform bipartite graph between $V_x$ and $V_y$. Then, for every $k$-colouring of the edges of $\Gamma$, there exists a $k$-colouring $\phi$ of the edges of $H$ and, for every $x \in V(H)$, a subset $U_x \subseteq V_x$ of order $|U_x| = \lambda n$ such that $(U_x, U_y; E_{\phi(xy)})$ is $(\eps, \frac{1}{2k}, p)$-dense for each $xy \in H$ where $E_{\phi(xy)} \subseteq E_{\Gamma}$ stands for the edges of colour $\phi(xy)$ between $U_x$ and $U_y$ in $\Gamma$. 
\end{lem}
\subsection{Key lemmas and proof of Theorem~\ref{thm:main-treewidth}}

As announced in the introduction, here we prove Theorem~\ref{thm:main-treewidth}. To do so, we will show the following two stronger lemmas, each of which implies one of the two statements in Theorem~\ref{thm:main-treewidth}. Both of them give bounds on the size-Ramsey number of bounded-degree graphs which are contained in a product of a bounded-degree tree and a clique of certain size. 
 We first state the two lemmas, and show how \Cref{thm:main-treewidth} is implied. In the next subsection, we give a proof of the first lemma, which relies on the proof of the classic lemma for embedding graphs into regular partitions. We are succinct with the details, as the embedding lemma is by now a standard tool (for an in-depth survey of it and other foundational regularity tools, see~\cite{komlos2002regularity}), and only minor modifications are necessary for our application.

After that, we give a proof of the second lemma, which relies on the ideas used for the first result, and on the embedding machinery developed in \cite{kohayakawa2011sparse}, with several changes needed to adapt their method to our setting. We present those changes in the appendix. 

\begin{lem} \label{lem:smallblowupsizeramsey}
Let $\Delta, d, k$ be positive integers and let $H$ be an $n$-vertex graph with $\Delta(H)\leq \Delta$, such that $H\subset T\boxtimes K_s$ for some $s=s(n)$ and a tree $T$ with $\Delta(T)\leq d$ and $v(T)=\tau = \tau(n)$. Then $\hat{r}_k(H)= O(\tau s^2)$.
\end{lem}
\begin{lem}\label{lem:blowupsizeramsey}
    Let $\Delta, d, k$ be positive integers and let $H$ be an $n$-vertex graph with $\Delta(H) \leq \Delta$, such that $H \subset T \boxtimes K_s$ for some $s = \Omega( e^{\sqrt{\log{n}}})$ and a tree $T$ with $\Delta(T) \leq d$ and $v(T) = \tau = \tau(n)$. Then $\hat{r}_k(H) = O(\tau s^2(\frac{\log s}{s})^{1/\Delta})$. 
\end{lem}    
    
\begin{proof}[ of \Cref{thm:main-treewidth}]
Let $H$ be an $n$-vertex graph of maximum degree $\Delta$ and treewidth $t=t(n)$. By Lemma~\ref{lem:embedintoproduct}, we know that $H\subseteq T\boxtimes K_s$ for $s=\Theta(t\log n)$ and $v(T)=O(n/s)$, with $\Delta(T)\leq d=d(\Delta)$.
Now, by Lemma~\ref{lem:smallblowupsizeramsey}, the size-Ramsey number of $H$ is $O(ns)=O(nt\log n)$, which gives the first part of the theorem.

For the second part, assume $t=\Omega(e^{\sqrt{\log{n}}})$. Note that by \Cref{lem:blowupsizeramsey} and since $s=\Theta(t\log n)=\Omega(e^{\sqrt{\log{n}}})$, we have $\hat{r}(H)= O(ns(\frac{\log s}{s})^{1/\Delta})=\tilde O(nt^{1-1/\Delta})$, as required.
\end{proof}    
\subsection{Proof of Lemma~\ref{lem:smallblowupsizeramsey}}
In this section, we give the proof of our first key lemma. We build our host graph in a way that is convenient for the structure $T \boxtimes K_s$ we know our graph $H$ has. Namely, we take a graph $\cR$ that is Ramsey for a constant blow-up of $T$ and blow it up by a factor of roughly $s$. We then make use of Lemma~\ref{lem:get-regularity} to get colorful lower-regular pairs within each random bipartite graph, thus yielding an auxiliary coloring of $\cR$. Next, using that we picked $\cR$ to be Ramsey for a constant blow-up of $T$, we find a monochromatic copy of the latter in the auxiliary coloring, which corresponds to a monochromatic $T$-shaped structure of lower-regular pairs in our host graph. We can then use an approach similar to the embedding lemma to embed $H$ into that monochromatic structure, being guided by the containment of $H$ in $T \boxtimes K_s$ to map each vertex to an appropriate set.

\begin{proof}[ of~\Cref{lem:smallblowupsizeramsey}]
Set $C^{-1}\ll \lambda\ll \varepsilon\ll k^{-1}, \Delta^{-1},d^{-1}$.
Let $F=T\boxtimes K_{\Delta+1}$ and note that $F$ has treewidth at most $2\Delta + 2$ by the definition of treewidth. 
Let $\mathcal R$ be the graph given by Theorem~\ref{cor:host-graph-bounded-treewidth} which is $k$-Ramsey for $F$ and which has $O(v(F))=O(\tau)$ vertices and constant maximum degree. The host graph for $H$ which we will consider is $G=\cR\boxtimes K_{Cs}$. Note that $G$ has $O(\tau s^2)$ edges. For each $x\in V(\cR)$, denote the copy of $K_{Cs}$ corresponding to $x$ by $V_x$. Consider an arbitrary $k$-coloring of the edges of $ G$.

We apply \Cref{lem:get-regularity} with $p=1$ and $\lambda$ to the graph $\Gamma:=G$, which is a blow-up of $\cR$. Hence, we get a coloring of the edges of $\cR$ and, for each $x\in V(\cR)$, a subset $U_x\subseteq V_x$ of size $\lambda Cs$, such that for every edge $\{x,y\}$ in $\cR$, the edges between $U_x$ and $U_y$ form an $(\varepsilon, \frac{1}{2k}, 1)$-dense pair in the color of the edge $\{x,y\}$ (meaning that the density between every two sets $X\subseteq U_x$ and $Y\subseteq U_y$ of sizes $|X|\geq \varepsilon|U_x|$ and  $|Y|\geq \varepsilon|U_y|$ is at least $\frac{1}{2k}-\varepsilon$). 

Now, since $\mathcal R$ is $k$-Ramsey for $F$, there is a monochromatic copy of $F$ in $\mathcal R$, which in turn means that for every edge $\{x,y\}$ of that copy, the pair $U_x$ and $U_y$ forms a $(\varepsilon, \frac{1}{2k}, 1)$-dense
pair in the same color, say red. As a last step before describing the embedding process, for each $v\in V(T)$, denote with $H_v$ the subgraph of $H$ which lives inside of the copy of $K_s$ that corresponds to the vertex $v$ (recall that $H\subseteq T\boxtimes K_s$). Since $H$ has maximum degree $\Delta$, it is $(\Delta+1)$-partite, so denote by $H_v^1,\ldots, H_v^{d+1}$ one such partition of $H_v$, for each $v\in V(T)$. 
To summarize, we now want to embed $H\subseteq T\boxtimes K_s$ into the  red subgraph of $G$ which is obtained from $T\boxtimes K_{\Delta+1}\boxtimes K_{\lambda Cs}$ by replacing the complete bipartite graphs between the relevant copies of $K_{\lambda Cs}$ by $(\varepsilon, \frac{1}{2k}, 1)$-dense pairs.

The embedding of $H$ into the found red subgraph of $G$ is done by a standard procedure, which is a variant of the well-known embedding lemma. Briefly, for each $v\in V(T)$, we will want to embed $H_v$ into $v\boxtimes K_{\Delta+1}\boxtimes K_{\lambda Cs}$, so that each $H_v^i$ is embedded into its own copy of $K_{\lambda Cs}$ in $v\boxtimes K_{\Delta+1}\boxtimes K_{\lambda Cs}$. This embedding can be done vertex by vertex, crucially maintaining the following invariant. For each vertex $x \in V(H)$ which has been assigned to the copy $U$ of $K_{\lambda Cs}$ but has not yet been embedded, we consider the `candidate set' $C_x$ that consists of all the vertices in $U$ that are in the common neighbourhood of all already embedded neighbours of $x$ in $H$. While embedding $H$, we ensure that for each such $x$, the size of the candidate set $C_x$ is at least $(\frac{1}{4k})^i|U|$, where $i$ is the number of neighbours of $x$ in $H$ which are already embedded at the observed point of the embedding process. When it is time for vertex $x$ to be embedded, we choose a vertex $y \in C_x$ for that purpose, in such a way that $y$ has enough neighbours in the candidate set of each neighbour $x'$ of $x$ in $H$ yet to be embedded, in order to preserve the aforementioned invariant. This is possible because, due to the lower-regular properties of the graph between $C_x$ and $C_{x'}$, at most $\eps \lambda Cs$ of the vertices in $C_x$ have degree less than $\frac{1}{4k} |C_{x'}|$ into $C_{x'}$. Additionally, at most $s$ many vertices in $C_x$ are already taken by vertices previously embedded in $U$, leaving us with at least $$|C_x| - \Delta \eps \lambda Cs - s \geq \Big(\frac{1}{4k}\Big)^{\Delta} \lambda C s - \Delta \eps \lambda Cs - s \gg 1  $$ viable candidates for the embedding of $x$, where we used that $\Delta(H)\leq \Delta$ and $(\frac{1}{4k})^{\Delta}\gg \varepsilon$. Thus, we can continue doing this until we embed the whole graph $H$. For more details, we refer the reader to \cite{chvatal1983ramsey} where such an embedding lemma is proven; the difference here is that we have more than a constant number of large cliques in which we want to embed, but the proof is essentially  the same.\end{proof}

\subsection{Proof of Lemma~\ref{lem:blowupsizeramsey}
}
We make use of the following result, which allows us to embed any bounded-degree subgraph of a blow-up of a bounded-degree graph $\cR$ into a certain blow-up of $\cR$ where the bipartite graphs between the parts are lower-regular pairs as in Definition~\ref{def:densepair}. Although here we only apply the lemma below with $\cR$ being a tree, we state it in this more general form, as it might be of independent interest and may prove to be useful in other settings which deal with embedding bounded-degree graphs.

\begin{restatable}{lem}{KRSSgeneralizationlem}
\label{lem:KRSS_generalization}
Let $\Delta,d, \rho^{-1} \ll \eps^{-1} \ll C' \ll C$ and $n$ be large enough. Let $s = \Omega(e^{\sqrt{\log{n}}})$ and $\frac{n}{s} \leq \tau \leq n$. Set $p:=C \big( \frac{\log{s}}{s} \big)^{1/\Delta}$ and $\Tilde{\Delta} := \Delta^4 + 2\Delta +1$. Then w.h.p. $G \sim G(\tau \Tilde{\Delta}C's , p)$ is such that the following holds.
Let $H$ be an $n$-vertex graph of maximum degree $\Delta$ that is a subgraph of $\cR \boxtimes K_s$ for some $\tau$-vertex graph $\cR$ with $\Delta(\cR) \leq d$. Let $J:= \cR \boxtimes K_{\Tilde{\Delta}}$ and $F \subseteq G$ be a graph with vertex set $V(F) = \bigcup_{x \in V(J)} F_x$ such that $|F_x| = C' s$ for each $x \in V(J)$ and edge set consisting of an $(\eps, \rho, p)$-dense bipartite graph between $F_x$ and $F_y$ for each $xy \in E(J)$. Then $H$ can be embedded in $F$. 
\end{restatable}
The proof of Lemma~\ref{lem:KRSS_generalization} mostly follows the approach from~\cite{kohayakawa2011sparse}, with some changes needed for our setting. We defer it to the appendix.

We are now ready to present the proof of our second key lemma. It makes use of the same main ideas as the first one, except that to get the tighter bound, the complete bipartite graphs between the sets of the host graph are now replaced by random graphs with smaller than constant density $p$. This drop in density then requires a much more careful treatment of the embedding process of $H$, for which the embedding lemma no longer suffices. This is where Lemma~\ref{lem:KRSS_generalization} proves to be helpful.
\begin{proof}[ of \Cref{lem:blowupsizeramsey}]
Let $H$ be as given above. We shall show that there exists a graph $G$ with $O(\tau s^{2-1/\Delta}\log^{1/\Delta}s)$ many edges such that $G \rightarrow_k H$.

It will be convenient to define the following constants.
Let $\Tilde{\Delta} := \Delta^4 + 2\Delta +1$ and let 
$$\Delta, d, k \ll d' \ll \eps^{-1} \ll C' \ll C.$$ 
Let $p = C\big(\frac{\log{s}}{s}\big)^{1/\Delta}$. Let $\lambda$ be as given by Lemma~\ref{lem:get-regularity} with $k:=k$, $\eps:=\eps$, and $\Delta := d'$.

Let $\cR$ be the graph that is $k$-Ramsey for all $\tau \Tilde{\Delta}$-vertex graphs of maximum degree $\Tilde{\Delta}(d+1)$ and maximum treewidth $2\Tilde{\Delta}$ given by Theorem~\ref{cor:host-graph-bounded-treewidth}. In particular, $\cR \rightarrow_k T\boxtimes K_{\Tilde{\Delta}}$, since by the definition of treewidth, $T\boxtimes K_{\Tilde{\Delta}}$ has treewidth at most $2\Tilde{\Delta}$. Note that $\cR$ has $O(\tau)$ vertices and constant maximum degree, say $d'$. We define our \emph{host graph} $G$ to be obtained by replacing each $v \in V(\cR)$ with an independent set $V_v$ on $\lambda^{-1}C's$ vertices and each edge $uv \in E(\cR)$ with a random bipartite graph between $V_u$ and $V_v$ in which each edge is added independently at random with probability $p$. 

Now consider an arbitrary colouring of the edges of $G$ in $k$ colours. We will show that $H$ can be embedded into one of the colour classes of $G$.
By a standard application of Chernoff bounds and a union bound, with high probability each random bipartite graph between $V_u$ and $V_v$ with $uv \in E(\cR)$ is $(\lambda,p)$-uniform. We fix an outcome of $G$ for which this is the case. In particular, this also implies that $e(G) = O(\tau s^{2-1/\Delta}\log^{1/\Delta}s)$.

We can now apply Lemma~\ref{lem:get-regularity} with $\Gamma := G$, $H:= \cR$, $n:=\lambda^{-1}C's$, and $p:=p$. This yields a $k$-colouring $\phi$ of the edges of $\cR$ and, for every $x \in V(\cR)$, a subset $U_x \subseteq V_x$ of size $|U_x| = \lambda |V_x|= C's$ such that for each $xy \in E(\cR)$, the graph $(U_x, U_y; E_{\phi(xy)})$ is $(\eps, \frac{1}{2k}, p)$-dense, where $E_{\phi(xy)} \subseteq E_G$ are the edges in colour $\phi(xy)$.

Next, note that by Theorem~\ref{cor:host-graph-bounded-treewidth}, the $k$-colouring $\phi$ of the edges of $\cR$ contains a monochromatic copy of $T \boxtimes K_{\Tilde{\Delta}} =: J$. This implies that $G$ contains a monochromatic subgraph $F$ with vertex set $V(F) = \bigcup_{x \in V(J)} F_x$ with $|F_x| = C's$ for each $x \in V(J)$ and edge set consisting of an $(\eps, \frac{1}{2k}, p)$-dense bipartite graph between $F_x$ and $F_y$ for each $xy \in E(J)$.

We can now apply Lemma~\ref{lem:KRSS_generalization} with $\cR:=T$ and $\rho:=\frac{1}{2k}$, concluding that the graph $F$ contains $H$ as a subgraph.
Since the copy of $F$ in $G$ was monochromatic, this concludes the proof of the lemma. 
\end{proof}

\paragraph{Acknowledgments}
We thank David Wood for asking for the bounds on the size-Ramsey number of bounded-degree planar graphs, which led to this paper~\cite{wood2022personal}.

\appendix
\section{Appendix}
\label{sec:appendix}
Here we give the proof of Lemma~\ref{lem:KRSS_generalization}, which roughly follows~\cite{kohayakawa2011sparse}. We begin by establishing certain useful `uniformity' properties, which allow us to proceed with the embedding later. First we define two classes of `bad' tripartite graphs, $\cB^{I}_p$ and $\cB^{II}_p$, where density of pairs $(X,Y), (Y,Z)$ is not inherited on one- respectively two-sided neighborhoods.\footnote{The definition slightly deviates from Definition 13 in \cite{kohayakawa2011sparse}, where the pairs $(X,Y)$ and $(Y,Z)$ are both $(\eps,\alpha,p)$-dense. Here, for technical reasons, we include the more general case of $(\eta,\alpha,p)$-dense pairs $(X,Y)$ whereby $\eta$ can differ from $\eps$.}
\begin{defn}
\label{def:bad}
    Let integers $m_1, m_2, m_3$ and reals $\alpha, \eps', \eps, \mu >0$, $\eta>0$, and $0<p\leq 1$ be given. 
    \begin{enumerate}
        \item Let $\cB^{I}_p(m_1, m_2,m_3, \alpha, \eps', \eps, \mu, \eta)$ be the family of tripartite graphs with vertex set $X \dot\cup Y \dot\cup Z$, where $|X| = m_1$, $|Y|=m_2$, and $|Z|=m_3$, satisfying
        \begin{enumerate}
            \item $(X,Y)$ is a $(\eta, \alpha,p)$-dense pair, 
            \item $(Y,Z)$ is a $(\eps, \alpha, p)$-dense pair, and
            \item there exists $X' \subseteq X$ with $|X'| \geq \mu |X|$ such that for every $x \in X'$ the pair $(N(x) \cap Y, Z)$ is not $(\eps', \alpha,p)$-dense.
        \end{enumerate}
        \item Let $\cB^{II}_p(m_1, m_2,m_3, \alpha, \eps', \eps, \mu, \eta)$ be the family of tripartite graphs with vertex set $X \dot\cup Y \dot\cup Z$, where $|X| = m_1$, $|Y|=m_2$, and $|Z|=m_3$, satisfying
        \begin{enumerate}
            \item $(X,Y)$ and $(X,Z)$ are $(\eta, \alpha,p)$-dense pairs,
            \item $(Y,Z)$ is a $(\eps, \alpha, p)$-dense pair, and
            \item there exists $X' \subseteq X$ with $|X'| \geq \mu |X|$ such that for every $x \in X'$ the pair $(N(x) \cap Y, N(x) \cap Z)$ is not $(\eps', \alpha,p)$-dense.
        \end{enumerate}
    \end{enumerate}
\end{defn}
Graphs that do not contain large subgraphs --- induced or not --- from the two above `bad' families are then said to have a denseness property $\cD^{\Delta}_{N,p}$, which we make precise in the next definition.
\begin{defn}
    For integers $N$ and $\Delta \geq 2$ and reals $\alpha, \gamma, \eps', \eps, \mu, \eta>0$ and $0 < p \leq 1$ we say that a graph $G = (V,E)$ with $V = [N]$ has the \emph{denseness property} 
    $\cD^{\Delta}_{N,p}(\gamma, \alpha, \eps', \eps, \mu, \eta)$,
    if $G$ contains no member from
    $$ B^I_p(m^I_1, m^I_2, m^I_3, \alpha, \eps', \eps, \mu, \eta) \cup  B^{II}_p(m^{II}_1, m^{II}_2, m^{II}_3, \alpha, \eps', \eps, \mu, \eta)$$
    with $m^I_1, m^I_3 \geq \gamma p^{\Delta-1} N$ and $m_2^I, m_1^{II}, m_2^{II}, m_3^{II} \geq \gamma p^{\Delta -2} N$ as a (not necessarily induced) subgraph.
\end{defn}

The following key technical lemma which is Corollary 16 in \cite{kohayakawa2011sparse} (and follows by repeatedly applying Proposition 15 therein) guarantees that random graphs $G(N,p)$ enjoy this denseness property at various parameter scales, whenever $p \gg (\frac{\log N}{N})^{\frac{1}{\Delta}} $. It will not be enough for our purposes that this fails to hold with any probability $o(1)$. However, performing the calculations explicitly allows us to bound the failure probability precisely. The other adjustments in the proof will stem from the changes in the definition of the `bad' graphs.

\begin{lem}[Corollary 16 in \cite{kohayakawa2011sparse}]
    \label{lem:regularity-inheritance}
    For all integers $\Delta, \Delta' \geq 2$ and all reals $\alpha, \mu, \gamma, \eps^* >0$, there exist $C>1$, $\eta = \eta(\Delta, \alpha, \mu)$, and $\eps_0, \ldots,\eps_{\Delta'}$ satisfying $0 < \eps_0 \leq \ldots \leq \eps_{\Delta'} \leq \eps^*$ such that if $p > C(\log{N}/N)^{1/\Delta}$, then $Pr[G(N,p) \in \cap_{k=1}^{\Delta'} \cD^{\Delta}_{N,p}(\gamma, \alpha, \eps_k, \eps_{k-1}, \mu, \eta)] \geq 1 - N^{-\Theta(Np^{\Delta-2})}$.
\end{lem}
\begin{proof}
    The proof is almost the same as that of Corollary 16 in \cite{kohayakawa2011sparse}, which can be summarized roughly as follows. Restricting to the case of one-sided regularity inheritance, one first considers the special case of a bad graph $T$ where the two parts $X$ and $Z$ are of fixed size, namely only a $p$-fraction of the size of $Y$ (the more general case can be reduced to this one).  The bipartite subgraphs $T[X,Y]$ and $T[Y,Z]$ are quite dense by assumption. Now consider the violating set $X' \subset X$, whose members $x \in X'$ all fail to produce $(\eps',\alpha,p)$-dense pairs $(N_T(x)\cap Y,Z)$. We restrict our attention to $X'' \subseteq X'$ such that each $x\in X''$ has `large' degree into $Y$. Because $T[X,Y]$ is $(\eta,\alpha,p)$-dense, $X''$ contains at least half of the vertices in $X'$. Now one can show that each such vertex $x$ produces a set $Y'_x \subset N_T(x) \cap Y$ of size $\frac{\eps' \alpha p m}{2}$ which, together with $Z$ produces a pair $(Y'_x,Z)$ of density strictly less than $\alpha-\eps'$. Any superset $Y_x$ of $Y'_x$ in $Y$ with size $\frac{\alpha p m}{2}$ then produces a violating pair, that is, the pair $(Y_x,Z)$ is not $(\eps',\alpha,p)$-dense. By repeating this procedure for all $x \in X''$, we could produce an entire family of violating pairs, which is unlikely to occur in $G(N,p)$. The latter can be proven by a union bound via counting the ways of choosing the vertices in the involved sets $X'', Y, Z$ and the edges in the bipartite graph $T[Y,Z]$ as well as the choices of the sets $Y_x$, which are limited by sparse regularity inheritance from $T[Y,Z]$ (cf. \cite{gerke2007smallt} Theorem 3.6). \par
    In what follows we note the necessary changes to Section 3.3 \cite{kohayakawa2011sparse}.
    \begin{itemize}
        \item In the statement of Proposition 15, after $\eps = \eps(\Delta, \alpha, \eps', \mu) >0$, add "there exists $\eta = \eta (\Delta, \alpha, \mu) > 0$" and change
        $$ Pr[G(N,p) \in \cD^{\Delta}_{N,p}(\gamma, \alpha,.\eps_k, \eps_{k-1}, \mu)] = 1 - o(1)$$
        to 
        $$ Pr[G(N,p) \in \cD^{\Delta}_{N,p}(\gamma, \alpha, \eps_k, \eps_{k-1}, \mu, \eta)] \geq 1 - N^{-\Theta(Np^{\Delta-2})}.$$
        \item In the proof of Corollary 16, add the argument $\eta$ to $\cD^{\Delta}_{N,p}(\gamma, \alpha, \eps_k, \eps_{k-1}, \mu)$.
        \item Change the definition of $\hat{\cD}^{\Delta}_{N,p}$ to include $\eta$ as an argument, in agreement with $\cD^{\Delta}_{N,p}$.
        \item In Proposition 17, add "for each $0<\eta <\frac{\mu}{100}, \frac{\alpha}{100}$" and  change the probability condition to $Pr[G(N,p) \in \hat{\cD}^{\Delta}_{N,p}(\gamma, \alpha, \eps',\eps, \mu, \eta)] \geq 1- N^{-\Theta(Np^{\Delta-2})}$. In particular, $\eta <\frac{\mu}{100}$ ensures that $|X''|$ is large enough, see the change below on the lower bound at the top of page 5052.
        \item In the proof of Proposition 17:
            \begin{itemize}
                \item Choose $C$, which was originally set to $C=(\frac{4}{\gamma})^{\frac{1}{\Delta}}$, to be, say, $100$ times larger. This will ensure the desired rate of decay later on in the proof, where the failure probability is estimated by a bound where $p$, which depends linearly on $C$, enters as $p^2$ in the exponent. See also the change on page 5051 indicated below.
                \item Add $\eta$ as an argument to $\cB^{I}_p$ and $\cB^{II}_p$ everywhere since $T[X,Y]$ is $(\eta,\alpha,p)$-dense according to Definition \ref{def:bad} instead of $(\eps,\alpha,p)$-dense as in \cite{kohayakawa2011sparse}.
                \item On page 5050, "Because of the assumption on $T$, the bipartite subgraphs $T[X,Y]$ and $T[Y,Z]$ of $T$ contain at least $(\alpha-\eps)p^2m^2$ edges each" --- this continues to hold for $T[Y,Z]$ but is no longer true of $T[X,Y]$, instead it contains at least $(\alpha-\eta)p^2m^2$ edges.
                \item On page 5050, "From the $(\eps, \alpha,p)$-denseness ..." towards the bottom of the page changes to "From the $(\eta, \alpha,p)$-denseness ..." and $|X''| \geq (1-\eps/\mu) |X'|$ changes to $|X''| \geq (1-\eta/\mu) |X'|$.
                \item On page 5051, we get that the probability that $T[X'', Y,Z]$ appears in $G(N,p)$ is at most $N^{-\Theta(Np^{\Delta-2})}$ as a result of the new choice of $C$.
                \item On top of page 5052, it is no longer true that $T[X,Y]$ and $T[X,Z]$ have at least $(\alpha-\eps)pm^2$ edges each, they have instead at least $(\alpha-\eta)pm^2$ edges each.
                \item On top of page 5052, in the lower bound for $|X''|$, we get instead $(1-2\eta/\mu)|X'| \geq |X'|/2$, from the $(\eta,\alpha,p)$-denseness of $T[X,Y]$ and $T[X,Z]$.
                \item On page 5052, we get that the probability that a graph from $\cB_p^{II}(m,\alpha,\eps',\eps,\mu,\eta)$ is contained in $G(n,p)$ is at most $N^{-\Theta(Np^{\Delta-2})}$.
            \end{itemize}
        \item In the proof of Proposition 15:
        \begin{itemize}
            \item Change the third sentence to "Indeed, roughly speaking, we show that each 'bad' tripartite graph $T \in \cB_p^{II}(m_1, m_2, m_3, \alpha, \eps', \eps, \mu,\eta)$ with arbitrary $m_1,m_2,m_3 \geq m$ contains a subgraph $\hat{T} \in \cB_p^{II}(m,\alpha,\eps'/2,\hat{\eps}, \mu/4, \hat{\eta})$ for some appropriate $\hat{\eps}$ and $\hat{\eta}$".
            \item In Claim 18, add $\hat{\eta}$ to the list of given positive reals, add "there exists $\eta = \eta(\hat{\eta}, \alpha, \mu)$", and add $\eta$ as an argument to $\cB_p^{II}(m_1, m_2,m_3,\alpha,\eps', \eps, \mu)$ and $\hat{\eta}$ as an argument to $\cB_p^{II}(m,\alpha,\eps'/2,\hat{\eps},\mu/4)$.
            \item In the paragraph after Claim 18, add "Pick $\hat{\eta} < \frac{\mu}{400}, \frac{\alpha}{400}$ for the application of Claim 18", and add $\hat{\eta}$ as an argument to $\cB^{I}_p$ and $\cB^{II}_p$; also change the probability to $1 - N^{-\Theta(Np^{\Delta-2})}$.
        \end{itemize}
        \item In the proof of Claim 18:
        \begin{itemize}
            \item Add $\hat{\eta}$ to the list of given constants, and take $\eta := \min\{\frac{\alpha}{400}, \frac{\mu}{400}, \eps_0(\alpha, \beta, \hat{\eta})\}$.
            \item On page 5053, in the definition of $T = (X \dot\cup Y \dot\cup Z, E_T)$, add $\eta$ as an argument to $\cB_p^{II}$.
            \item On page 5053, replace $\eps$ with $\eta$ everywhere in the sentence "Owing to the choice of $\eps$ ..." (we can do this since $\eta <\frac{\mu}{100}$).
            \item On the bottom of page 5053, change the last sentence of the penultimate paragraph to "We shall show that with positive probability $\hat{T}$ is from $\cB_p^{II}(m,\alpha, \eps'/2, \hat{\eps}, \mu/4, \hat{\eta})$.
            \item In the last paragraph of page 5053, replace $\hat{\eps}$ with $\hat{\eta}$ when it comes to pairs $(\hat{X}, \hat{Y})$ and $(\hat{X}, \hat{Z})$.
            \item On the bottom of page 5054, add $\hat{\eta}$ as an argument to $\cB_p^{II}(m,\alpha,\eps'/2,\hat{\eps},\mu/4)$.
        \end{itemize}
    \end{itemize}
\end{proof}

Let $G=(V,E)$ be a graph and $k \geq 1$ be an integer. We define an auxiliary bipartite graph $\Gamma(k, G) = ( \binom{V}{k} \dot\cup V, E_{\Gamma(k,G})$ by
$$ (K,v) \in E_{\Gamma(k,G)} \Longleftrightarrow  \{w,v\}\in E(G) \text{ for all } w\in K.$$
The next definition and lemma will help us conclude that if $G$ is the random graph $G(N,p)$, then $\Gamma(k, G)$ has no "dense parts".
\begin{defn}
Let integers $N$ and $k\geq 1$ and reals $\xi>0$ and $0<p\leq 1$ be given. A graph $G=(V,E)$ with $V=[N]$ has the \emph{congestion property} $\mathscr{C}^k_{N,p}(\xi)$ if for every $U \subseteq V$ and every family $\cF_k \subseteq \binom{V \setminus U}{k}$ of pairwise disjoint $k$-sets with
\begin{itemize}
    \item $|\cF_k| \leq \xi N$ and
    \item $|U| \leq |\cF_k|$
\end{itemize}
we have
$$ e_{\Gamma(k,G)}(\cF_k, U) \leq p^k |\cF_k||U| + 6 \xi N p^k |\cF_k|.$$
\end{defn}
Similarly to the denseness property from Lemma \ref{lem:regularity-inheritance}, we will require better control over the probability that a graph fails to have the congestion property than the original phrasing of Corollary 12 in \cite{kohayakawa2011sparse}. We will hence adjust the proof accordingly.

\begin{lem}[Corollary 12 in~\cite{kohayakawa2011sparse}]
\label{lem:auxiliary-graph-not-too-dense}
For every integer $\Delta \geq 1$ and real $\xi>0$, there exists $C>1$ such that if $p > C (log{N} / N)^{1/\Delta}$, then $Pr[G(N,p) \in \cap_{k=1}^{\Delta}\mathscr{C}_{N,p}^{k}(\xi) ] = 1- e^{-100\log^2{N}}$.
\end{lem}
\begin{proof}
    The proof is almost the same as that of Corollary 12 in~\cite{kohayakawa2011sparse}. In what follows we note the necessary changes in Section 3.2.
    \begin{itemize}
        \item In the statement of Proposition 11, change $o(1)$ to $e^{-200\log^2{N}}$.
        \item At the end of the analysis for Case 1 in the proof of Proposition 11, note that the probability of the bad event can in fact be upper bounded by $e^{-2N}$.
        \item At the end of Case 2 in the proof of Proposition 11, one can upper bound the probability of the bad event by $e^{-400\log^2{N}}$, for $C$ large enough.
    \end{itemize}
\end{proof}

We are now ready to prove \Cref{lem:KRSS_generalization}, which we restate here for convenience of the reader.
\KRSSgeneralizationlem*

\begin{proof}[ of Lemma~\ref{lem:KRSS_generalization}]
We begin the proof by setting the constants. This will allow for the application of Lemmas \ref{lem:regularity-inheritance} and \ref{lem:auxiliary-graph-not-too-dense} at the necessary level. In particular, the subgraphs of our host graph will have denseness and congestion properties which still hold after taking union bounds. The second step will consist of preparing our graph $H$ for its embedding. Here we use the assumption that $H$ is a subgraph of $\cR\boxtimes K_s$. We first partition the vertices of $H$ according to the copy of $K_s$, indexed by the vertices of $\cR$, into which they are mapped. This ordered partition is then refined using a coloring (of the third power of $H^3$), ensuring that no two vertices at distance at most $3$ with respect to $H$ are mapped to the same class. The actual embedding of $H$ into $F$ is then constructed inductively, maintaining at each step $\ell$ for every not yet embedded vertex $z$ of $H$ a large enough candidate set $C_{\ell}(z)$ which takes into account all the already embedded neighbors of $z$, and ensuring that the candidate sets for pairs of not yet embedded vertices $z, z'$ that are neighbors in $H$ have candidate sets $C_{\ell}(z),C_{\ell}(z')$ which induce an appropriately regular pair in $F$. Given a vertex $y\in H$ and a candidate set $C_{\ell}(y)$, a naive embedding of $y$ into it will not work, as previously embedded vertices from the same class may have exhausted the set. Instead, we use Hall's condition to find one distinct vertex in $C_{\ell}(y)$ for each $y$ in the current vertex class of $H$ simultaneously. \par
\textbf{Notation and Constants.} We first fix some notation. In $J = \cR \boxtimes K_{\Tilde{\Delta}}$, for each $x \in V(\cR)$, we denote by $J_x$ the copy of $K_{\Tilde{\Delta}}$ that corresponds to $x$, and by $\mathbf{F}_x$ the copy of $K_{\Tilde{\Delta}} \boxtimes K_{C's}$ in $F$ that corresponds to $J_x$. 

Next, we choose the constants. Let $\eps_0 := \eps$, $\mu :=\frac{1}{4\Delta^2}$, $\xi:= \frac{1}{(d+1)\Tilde{\Delta}C'}$ and $\gamma:=\frac{1}{(2/\rho)^{\Delta-1}3\Tilde{\Delta}}$. Let $\eta$ be as given by Lemma~\ref{lem:regularity-inheritance} applied with $\Delta:=\Delta$, $\alpha := \rho$, $\mu:=\mu$ and 
$$ 0 < \eps_0 \leq \ldots \leq \eps_{2\Delta} \leq \eps^* := \min\{\frac{\rho}{6 \Tilde{\Delta}}, \eta\}.$$
be as given by Lemma~\ref{lem:regularity-inheritance} applied with $\Delta:=\Delta$, $\Delta' := 2\Delta$, $\alpha := \rho$, $\mu:=\mu$, $\gamma:=\gamma$. 
The value of $\eps^*$ ensures in particular that in the embedding which we will construct step by step, the candidate sets for not yet embedded vertices are large enough. The choice of $\gamma$ will ensure that the candidate sets are large enough to form dense pairs. The choice of $\xi$ will enable us to verify Hall's condition.

For each distinct $x,y,z \in V(\cR)$ such that $xy,yz \in E(\cR)$, apply Lemma~\ref{lem:regularity-inheritance} with $\Delta:=\Delta$, $\Delta' := 2\Delta$, $\alpha := \rho$, $\mu:=\mu$, $\gamma:=\gamma$, $\eps^*:=\eps^*$ and $N:=3 \Tilde{\Delta}C's$ to $F[\mathbf{F}_{x} \cup \mathbf{F}_{y} \cup \mathbf{F}_{z}]$, which is a subgraph of $G(N,p)$. By a union bound over all at most $O(t)$ such triples $x,y,z \in V(\cR)$, we conclude that w.h.p. for each $xy, yz \in E(\cR)$, we have $F[\mathbf{F}_{x} \cup \mathbf{F}_{y} \cup \mathbf{F}_{z}]\in \cap_{k=1}^{2\Delta} \cD^{\Delta}_{N,p}(\gamma, \rho, \eps_k, \eps_{k-1}, \mu, \eta)$, since
$$O(t) s^{-\Theta(sp^{\Delta-2})} = O(n) e^{-\Theta(sp^{\Delta-2} \log{s})} = O(n) e^{-\Theta(s^{2/\Delta} \log^{(2\Delta-2)/\Delta}{s})} = O(n)e^{-\omega(\log{n})} = o(1).$$
From now on we condition on this being the case.

Similarly, for each $x \in V(\cR)$, we apply Lemma~\ref{lem:auxiliary-graph-not-too-dense} with $\Delta:=\Delta$, $\xi:=\xi$ and $N:=\big(\mathsf{deg}_{\cR}(x)+1\big) \Tilde{\Delta}C's$ to $F[\bigcup_{y \in N_{\cR}(x)} \mathbf{F}_y \cup \mathbf{F}_x]$, which we can do since it is a subgraph of $G(N,p)$. We next apply a union bound over all $t$ such vertices $x \in V(\cR)$, to show that the probability that for some $x \in V(\cR)$, we have $F[\bigcup_{y \in N_{\cR}(x)} \mathbf{F}_y \cup \mathbf{F}_x] \notin \cap_{k=1}^{\Delta} \mathscr{C}_{N,p}^k(\xi)$ is at most
$$ t e^{-100\log^2(2\Tilde{\Delta}C's)} \leq n e^{-100\log^2{s}} \leq ne^{-100\log^2(\Omega(e^{\sqrt{\log{n}}}))} \leq n e^{-25\log{n}} = o(1).$$
From here on we condition on this event not occurring.

\textbf{Preparation of $H$.} We start by preparing $H$ for the embedding. For this, consider an embedding $\psi: H \rightarrow \cR\boxtimes K_s$. We partition $V(H)$ into classes $\{S_x\}_{x \in V(\cR)}$ with $v \in S_x$ for $x \in V(\cR)$ if $\psi(v)$ belongs to the copy of $K_s$ in $\cR \boxtimes K_s$ that corresponds to $x$. We denote $H_x:=H[S_x]$. Let $q$ be an arbitrary "root" of $\cR$ and consider a breadth-first ordering of $V(\cR)$ starting from $q$, namely $q=x_1, \ldots, x_t$. We will embed $S_x$ into $\mathbf{F}_x$ for each $x = x_1, \ldots, x_t$ in this order. For each $x \in V(\cR)$, let $p_1(x),  \ldots, p_{d_x}(x)$ be the "parents" of $x$, namely those of its neighbours that come before $x$ in the order $x_1, \ldots, x_t$. Note that $0 \leq d_x \leq d$. With slight abuse of notation, denote $S_{p(x)} := \bigcup_{i=1}^{d_x} S_{p_i(x)}$.

Consider the third power $H^3$ of $H$, that is, $uv \in E(H^3)$ if and only if $u\neq v$ and there is a path between $u$ and $v$ in $H$ consisting of at most $3$ edges. Note that $\Delta (H^3)\leq \Delta+\Delta(\Delta-1)+\Delta(\Delta-1)^2 = \Delta^3 - \Delta^2 + \Delta $ and therefore $\chi(H^3) \leq \Delta^3 - \Delta^2 + \Delta + 1$. Fix some $x \in V(\cR)$ and let $f_x$ be a $\Delta^3 - \Delta^2 + \Delta + 1$-vertex colouring of $H^3[S_x]$. Then $f_x$ partitions $S_x$ into $\Delta^3 - \Delta^2 + \Delta + 1$ classes such that if two vertices $u,v$ are in the same class, they have distance at least $4$ in $H$, and so there are no edges between $N_H(u)$ and $N_H(v)$. Next, we refine this partition depending on the "left-degrees" of the vertices. We say that $u, v \in S_x$ are equivalent if $f_x(u) = f_x(v)$ and 
$$ |N_H(u) \cap \Big( \{w \in S_x : f_x(w) < f_x(u)\}\cup S_{p(x)} \Big) | = |N_H(v) \cap \Big(\{w \in S_x : f_x(w) < f_x(v)\} \cup S_{p(x)}\Big) |.$$
Note that vertices of $S_x$ will be embedded in order of increasing colour according to $f_x$. Thus, two vertices in $S_x$ are equivalent if they have the same colour by $f_x$ and the same number of neighbours that will be embedded before them. 
This equivalence relation partitions $S_x$ into at most $(\Delta^3 - \Delta^2 + \Delta + 1)(\Delta+1) = \Tilde{\Delta}$ many classes. Denote the partition classes by $W^x_1, \ldots, W^x_{\Tilde{\Delta}}$, some of which may be empty, and let $g_x:S_x \rightarrow [\Tilde{\Delta}]$ be the corresponding partition function, that is, for $w \in S_x$, we have $g_x(w) = j$ if and only if $w \in W^x_j$. Note that $W^x_1, \ldots, W^x_{\Tilde{\Delta}}$ are ordered in increasing order of $f_x$. Each $W^x_i$ will be embedded into a different $F_y$ with $y \in J_x$, but for simplicity we now rename these vertex classes so that we can treat all of them at the same time.

Set $Q := t \Tilde{\Delta}$. Then let $W_1, \ldots, W_{Q}$ be obtained by concatenating the sequences $(W^{x_1}_1, \ldots, W^{x_1}_{\Tilde{\Delta}}),$ $ \ldots, (W^{x_t}_1, \ldots, W^{x_t}_{\Tilde{\Delta}})$ in this order, and again let $g: V(H) \rightarrow [Q]$ be the corresponding partition function, that is, $g(w) = j$ if and only if $w \in W_j$ for $w \in V(H)$. Thus, if $g(u) = g(v)$, then
$$|N_H(u) \cap \{w \in V(H): g(w) < g(u)\}|=|N_H(v) \cap \{w \in V(H): g(w) < g(v)\}|.$$
Let $u \in V(H)$ and $\ell \leq g(u)$. Then we define the \emph{left-degree} of $u$ with respect to $g$ and $\ell$ as
$$ \mathsf{ldeg}^{\ell}_g(u) := |N_H(u) \cap \{w \in V(H): g(w) \leq \ell \}|. $$

\textbf{Inductively embedding $H$ into $F$.} We proceed to embed $H$ into $F$. For this, we first relabel the vertex classes $\{F_y\}_{y \in V(J)}$ by ordering them as $F_1, \ldots, F_{Q}$, where in the order $x_1, \ldots, x_{t}$ of $V(\cR)$ that we fixed above, for each $x_i$ we have that $\{F_y\}_{y \in J_{x_i}}$ are relabelled as $F_{(i-1)\Tilde{\Delta} + 1}, \ldots, F_{i \Tilde{\Delta}}$ (in arbitrary order). We will proceed by induction, embedding $W_i$ into $F_i$ for each $i = 1, \ldots, Q$. Note that even though $J$ is not a complete graph, due to the fact that $H$ is a subgraph of $\cR \boxtimes K_s$ and the way we picked $W_1, \ldots, W_{Q}$ and $F_1, \ldots, F_{Q}$, we have that if $u \in W_i$ and $v \in W_j$ and $uv \in E(H)$, then there is an $(\eps_0, \rho, p)$-dense bipartite graph between $F_i$ and $F_j$.

Throughout our embedding, we make sure that the following statement $(\cS_{\ell})$ holds for $\ell = 0, 1, \ldots, Q$:
\begin{align*}
    (\cS_{\ell}) & \text{ There exists a partial embedding } \phi_{\ell} \text{ of } H[\cup_{j=1}^{\ell} W_j] \text{ into } F[\cup_{j=1}^{\ell} F_j] \text{ such that for every }\\
    &  z\in \cup_{j=\ell+1}^{Q}W_j  \text{ there exists a candidate set } C_{\ell}(z) \subseteq V(F) \text{ given by }\\
    & \text{(a) } C_{\ell}(z) = \cap \{N_F(\phi_{\ell}(x)): x \in N_H(z) \text{ and } g(x) \leq \ell \} \cap F_{g(z)}, \\
    & \text{satisfying} \\
    & \text{(b) } |C_{\ell}(z)| \geq \Big(\frac{\rho p}{2}\Big)^{\mathsf{ldeg}^{\ell}_g(z)}m \text{, where } m=|F_{g(z)}| = C's, \text{ and} \\
    & \text{(c) for every edge } \{z,z'\} \in E(H) \text{ with } g(z), g(z') > \ell \text{ the pair } (C_{\ell}(z), C_{\ell}(z')) \text{ is } \\
    & \hspace{0.6cm} (\eps_{\mathsf{ldeg}_g^{\ell}(z) + \mathsf{ldeg}_g^{\ell}(z')}, \rho, p)\text{-dense in } F.\\
\end{align*}

The inductive statement $(\cS_{\ell})$ tells us that the first $\ell$ classes $W_1, \ldots, W_{\ell}$ of $H$ can be embedded into $F_1 \cup \ldots \cup F_{\ell}$ in such a way that for every not yet embedded vertex $z \in W_{\ell+1} \cup \ldots \cup W_Q$, there is a large enough candidate set $C_{\ell}(z)$ that consists of the intersection of $F_{g(z)}$ and the neighbourhoods of all already embedded neighbours of $z$. The third condition of $(\cS_{\ell})$ additionally ensures that all pairs of candidate sets inherit regularity.

Note that from $(\cS_Q)$ it follows that $H$ can be embedded into $F$, which means that showing that $(\cS_0), \ldots, (\cS_{Q})$ hold completes the proof of Lemma~\ref{lem:KRSS_generalization}. We next prove that $(\cS_{\ell})$ holds for $\ell = 0,1, \ldots, Q$ by induction.

\textit{Basis of the induction:} $\ell=0$. We then have that $\phi_0$ is the empty embedding and $C_0(z) = F_{g(z)}$ for every $z \in V(H)$ by (a). This implies that property (b) also holds, as $|C_0(z)|=m$ and $\mathsf{ldeg}^0_g(z) = 0$. Property (c) holds since, as mentioned above, for each $z z' \in E(H)$, there is an $(\eps_0, \rho, p)$-dense graph between $F_{g(z)}=C_0(z)$ and $F_{g(z')} = C_0(z')$. 

\textit{Induction step:} $\ell \rightarrow \ell+1$. Suppose that $(\cS_{\ell})$ holds for some $\ell < Q$. We will now extend $\phi_{\ell}$ to an embedding $\phi_{\ell+1}$ that satisfies (a), (b), and (c) to show that $(\cS_{\ell+1})$ holds. We do this in the following way. Firstly, for each $y\in W_{\ell+1}$, we identify a subset $C(y) \subseteq C_{\ell}(y)$ such that if $y$ is embedded into some vertex from $C(y)$, then for every "right-neighbour" $z$ of $y$ in $H$, the candidate set $C_{\ell+1}(z):=C_{\ell}(z) \cap N_F(\phi_{\ell+1}(y))$ will be such that properties (b) and (c) will still hold.

However, since if $\mathsf{ldeg}^{\ell}_g \geq 1$, we have $|C(y)| \leq |C_{\ell}(y)| = o(s) \ll |W_{\ell + 1}|$, we cannot greedily pick an embedding of each $y$ into $C(y)$, as this way some $C(y)$ may be entirely occupied by other embeddings before we get to embed it. Thus, in a second step we will use Hall's condition to find distinct representatives for $\{C(y): y\in W_{\ell+1}\}$, and then we will set $\phi_{\ell+1}(y)$ to be the representative of $C(y)$. We now describe these two steps in detail.

Let $y \in W_{\ell+1}$, and set
$$ N_H^{\ell+1}(y) := \{ z \in N_H(y) : g(z) > \ell+1 \}. $$
We say that $v \in C_{\ell}(y)$ is \emph{bad} (that is, it will not be in $C(y)$) if there is some vertex $z \in N_H^{\ell+1}(y)$ such that the set $N_F(v) \cap C_{\ell}(z)$ does not satisfy condition (b) or (c) of $(\cS_{\ell+1})$ and therefore cannot become $C_{\ell+1}(z)$.

We start by focusing on (b) of $(\cS_{\ell+1})$. Let $z \in N_H^{\ell+1}(y)$. We know by (c) of $(\cS_{\ell})$ that $(C_{\ell}(y), C_{\ell}(z))$ is an $(\eps_{\mathsf{ldeg}_g^{\ell}(y) + \mathsf{ldeg}_g^{\ell}(z)}, \rho, p)$-dense pair. Thus, there are at most $\eps_{\mathsf{ldeg}_g^{\ell}(y) + \mathsf{ldeg}_g^{\ell}(z)}|C_{\ell}(y)| \leq \eps_{2\Delta} |C_{\ell}(y)|$ vertices $v$ in $C_{\ell}(y)$ such that
$$ |N_F(v) \cap C_{\ell}(z)| < \Big(\rho - \eps_{\mathsf{ldeg}_g^{\ell}(y) + \mathsf{ldeg}_g^{\ell}(z)} \Big)p |C_{\ell}(z)|.$$
Considering all $z\in N_H^{\ell+1}(y)$ in the same way, it follows that for all but at most $\Delta \eps_{2\Delta} |C_{\ell}(y)|$ vertices $v \in C_{\ell}(y)$, the following holds for all $z \in N_H^{\ell+1}(y)$:
$$ |N_F(v) \cap C_{\ell}(z)| \geq \Big(\rho - \eps_{2\Delta}\Big)p |C_{\ell}(z)| $$
$$ \geq \Big( \rho - \eps_{2\Delta} \Big)p \Big( \frac{\rho p}{2}\Big)^{\mathsf{ldeg}^{\ell}_g(z)} |F_{g(z)}| \geq \Big( \frac{\rho p}{2} \Big)^{\mathsf{ldeg}_g^{\ell+1}(z)}|F_{g(z)}|,$$
where for the penultimate inequality we used condition (b) of $(\cS_{\ell})$ and the final inequality follows from our choice of $ \eps^*\geq \eps_{2\Delta}$.

Next, we consider property (c) of $(\cS_{\ell+1})$. Let $e=\{z,z'\}$ with $g(z), g(z') > \ell+1$ and with at least one of $z, z'$ in $N_H^{\ell+1}(y)$. The number of these edges is at most $\Delta (\Delta-1) < \Delta^2$. Let $\tilde{y}, \tilde{z}, \tilde{z}' \in V(\cR)$ be such that $F_{g(y)} \subseteq \mathbf{F}_{\tilde{y}}, F_{g(z)} \subseteq \mathbf{F}_{\tilde{z}}, F_{g(z')} \subseteq \mathbf{F}_{\tilde{z}'}$, and let $j := \mathsf{ldeg}^{\ell}_g(z)+ \mathsf{ldeg}^{\ell}_g(z')$. If $z, z' \in N_H^{\ell}(y)$, we have
$$ \max \{\mathsf{ldeg}_g^{\ell}(y), \mathsf{ldeg}_g^{\ell}(z), \mathsf{ldeg}_g^{\ell}(z')\} \leq \Delta -2,$$
because each of $y,z,z'$ has at least two neighbours in $W_{\ell+1} \cup \ldots \cup W_Q$. Property (b) of $(\cS_{\ell})$ then implies
$$ \min \{|C_{\ell}(y), C_{\ell}(z), C_{\ell}(z')|\} \geq \Big( \frac{\rho p}{2} \Big)^{\max\{\mathsf{ldeg}_g^{\ell}(y),\mathsf{ldeg}_g^{\ell}(z),\mathsf{ldeg}_g^{\ell}(z')\}} C' s \geq \gamma p^{\Delta -2} 3 \Tilde{\Delta}C's,$$
by our choice of $\gamma$. Since either $\tilde{y} = \tilde{z}$ or $\tilde{y}\tilde{z} \in E(\cR)$ (and the same holds for $\tilde{z}, \tilde{z}'$ and $\tilde{y}, \tilde{z}'$), we know that $F[\mathbf{F}_{\tilde{y}} \cup \mathbf{F}_{\tilde{z}} \cup \mathbf{F}_{\tilde{z}'}]$  is a subgraph of some graph in $\cD_{3\Tilde{\Delta}C's,p}^{\Delta}(\gamma, \rho, \eps_{j+1}, \eps_j, \mu, \eta)$. Thus, there are at most $\mu |C_{\ell}(y)|$ vertices $v$ in $C_{\ell}(y)$ such that $\big(N_F(v) \cap C_{\ell}(z), N_F(v) \cap C_{\ell}(z')\big)$ is not $(\eps_{j+1},\rho,p)$-dense.

Suppose now that only one of $z,z'$ is in $N_H^{\ell+1}(y)$, say $z \in N_H^{\ell+1}(y)$ and $z' \notin N_H^{\ell+1}(y)$. Then we have
$$ \max \{\mathsf{ldeg}^{\ell}_g(y), \mathsf{ldeg}^{\ell}_g(z')\} \leq \Delta-1 \text{ and } \mathsf{ldeg}_g^{\ell}(z)\leq \Delta-2.$$
Therefore, analogously to above,
$$ \min\{|C_{\ell}(y)|, |C_{\ell}(z')|\} \geq \gamma p^{\Delta-1} 3\Tilde{\Delta}C's \text{ and } |C_{\ell}(z)| \geq \gamma p^{\Delta-2} 3\Tilde{\Delta}C's.$$
Again, since either $\tilde{y} = \tilde{z}$ or $\tilde{y}\tilde{z} \in E(\cR)$ (and the same holds for $\tilde{z}, \tilde{z}'$), $F[\mathbf{F}_{\tilde{y}} \cup \mathbf{F}_{\tilde{z}} \cup \mathbf{F}_{\tilde{z}'}]$  is a subgraph of some graph in $\cD_{3\Tilde{\Delta}C's,p}^{\Delta}(\gamma, \rho, \eps_{j+1}, \eps_j, \mu, \eta)$. Thus, there are at most $\mu |C_{\ell}(y)|$ vertices $v \in C_{\ell}(y)$ such that $\big( N_F(v)\cap C_{\ell}(z), C_{\ell}(z')\big)$ is not $(\eps_{j+1},\rho, p)$-dense.

Note that in both cases, $\mathsf{ldeg}^{\ell+1}_g(z) + \mathsf{ldeg}^{\ell+1}_g(z') \geq \mathsf{ldeg}^{\ell}_g(z) + \mathsf{ldeg}^{\ell}_g(z') + 1 = j+1$, and so any pair that is $(\eps_{j+1},\rho, p)$-dense is also $(\eps_{\mathsf{ldeg}^{\ell+1}_g(z) + \mathsf{ldeg}^{\ell+1}_g(z')},\rho, p)$-dense. For some fixed $v \in C_{\ell}(y)$, set $\hat{C}_{\ell}(z) = C_{\ell}(z) \cap N_F(v)$ if $z \in N_H^{\ell+1}(y)$ and $\hat{C}_{\ell}(z) = C_{\ell}(z)$ if $z \notin N_H^{\ell+1}(y)$, and define $\hat{C}_{\ell}(z')$ in the same way.

What we have shown so far is that there are at least
$$ (1 - \Delta \eps_{2\Delta} - \Delta^2 \mu) |C_{\ell}(y)| $$
vertices $v \in C_{\ell}(y)$ with the properties
\begin{align*}
    \text{(b') } & |N_H(v) \cap C_{\ell}(z)| \geq \Big( \frac{\rho p}{2} \Big)^{\mathsf{ldeg}_g^{\ell+1}(z)}|F_{g(z)}| \text{ for each } z \in N_H^{\ell+1}(y) \text{ and}\\
    \text{(c') } & \big(\hat{C}_{\ell}(z), \hat{C}_{\ell}(z')\big) \text{ is } (\eps_{\mathsf{ldeg}^{\ell+1}_g(z) + \mathsf{ldeg}^{\ell+1}_g(z')}, \rho,p)\text{-dense for all edges } \{z,z'\} \text{ of } H \text{ with }\\
    & g(z), g(z') > \ell +1 \text{ and } \{z,z'\} \cap N_H^{\ell+1}(y) \neq \varnothing.
\end{align*}

Denote the vertices in $C_{\ell}(y)$ that satisfy (b') and (c') by $C(y)$. For all $y,y'\in W_{\ell+1}$, since $g(y)=g(y')=\ell$, and by our choice of the partition $W_1, \ldots, W_Q$, we have that $\mathsf{ldeg}_g^{\ell}(y) = \mathsf{ldeg}_g^{\ell}(y')$.  Let
$$ k = \mathsf{ldeg}_g^{\ell}(y) \text{ for some } y \in W_{\ell+1}$$
and note that $0 \leq k \leq \Delta$. Then by our choice of $\mu$ and $\eps_{2\Delta} \leq \eps^{*}$ and from condition (b) of $(\cS_{\ell})$, we have
$$ |C(y)| \geq (1 - \Delta \eps_{2\Delta} - \Delta^2 \mu) |C_{\ell}(y)| \geq (1 - \Delta \eps_{2\Delta} - \Delta^2 \mu) \Big( \frac{\rho p}{2} \Big)^{k} C's \geq \frac{\rho^k}{2^{k+1}} p^k C' s.$$
This finishes the first part of our inductive step.

Next, we move to the second part, in which we show that distinct representatives for the system $\big( C(y)\big)_{y\in W_{\ell+1}}$ exist. For this, we use Hall's condition~\cite{diestel2017graph}, which tells us that in a bipartite graph with bipartitions $X$ and $Y$, there is a perfect matching if for every $X' \subseteq X$, it holds that $|N(X)| \geq |X|$.  Thus, in our case it is enough to show that for every $Y \subseteq W_{\ell+1}$, it holds that
$$ |Y| \leq \Big| \bigcup_{y\in Y} C(y) \Big|.$$
In case $1 \leq |Y| \leq \rho^k p^k C' s / 2^{k+1}$, the above holds due to our lower bound on each $|C(y)|$.

Now let $Y \subseteq W_{\ell+1}$ be a set with $|Y| > \rho^k p^k C' s / 2^{k+1}$. As mentioned above, for each $y\in W_{\ell+1}$, we have $\mathsf{ldeg}_g^{\ell}(y)=k$, so there is a $k$-tuple $K(y) = \{u_1(y), \ldots, u_k(y)\}=N_H(y) \setminus N_H^{\ell+1}(y)$ of neighbours of $y$ that are already embedded. Moreover, since we chose the partition $W_1, \ldots, W_Q$ in such a way that the distance in $H$ between any two vertices in the same class $W_i$ is at least $4$, we have that for every $y,y' \in W_{\ell+1}$, the sets $K(y)$ and $K(y')$ are disjoint. Therefore, the sets of already embedded vertices $\phi_{\ell}(K(y))$ and $\phi_{\ell}(K(y'))$ are disjoint too, and so $\cF_k=\{\phi_{\ell}(K(y)):y\in Y\} \subseteq \binom{V(F)}{k}$ is a family of disjoint sets of size $k$ in $V(F)$. Furthermore,
$$ C(y) \subseteq \bigcap_{v\in \phi(K(y))} N_F(v).$$
Set
$$ U = \bigcup_{y\in Y} C(y) \subseteq F_{\ell+1},$$
and observe that if $F_{\ell+1} \subseteq \mathbf{F}_{\tilde{y}}$ for some $\tilde{y} \in V(\cR)$, then $U \subseteq \mathbf{F}_{\tilde{y}}$ and
$$ \cF_k \subseteq \binom{(\bigcup^{d_{\tilde{y}}}_{i=1}\mathbf{F}_{p_i(\tilde{y})} \dot\cup \mathbf{F}_{\tilde{y}} ) \setminus U }{k}.$$
Assume for contradiction that 
$$ |U| < |Y| = |\cF_k|.$$
Since we conditioned on $\mathscr{G}:= F[\bigcup_{\tilde{n} \in N_{\cR}(\tilde{y})}\mathbf{F}_{\tilde{n}} \dot\cup \mathbf{F}_{\tilde{y}} ] \in \mathscr{C}^k_{(\mathsf{deg}_{\cR}(\tilde{y})+1)\Tilde{\Delta}C's,p}(\xi)$, and we have $|\cF_k|\leq |W_{\ell+1}|\leq s \leq \xi (\mathsf{deg}_{\cR}(\tilde{y})+1) \Tilde{\Delta} C' s$, it follows that
$$ e_{\Gamma(k,\mathscr{G})}(\cF_k, U) \leq p^k |\cF_k| |U| + 6 \xi (\mathsf{deg}_{\cR}(\tilde{y})+1)\Tilde{\Delta}C's p^k |\cF_k|.$$
From our lower bound on $|C(y)|$ above, it follows that
$$ e_{\Gamma(k,\mathscr{G})}(\cF_k, U) \geq \frac{\rho^k}{2^{k+1}} p^k C's |\cF_k|.$$
The last two inequalities together with $C' \gg \Delta, \rho^{-1},d$ imply that
$$ \Big| \bigcup_{y\in Y} C(y) \Big| = |U| \geq \Big( \frac{\rho^k}{2^{k+1}} - 6(\mathsf{deg}_{\cR}(\tilde{y})+1) \xi \Tilde{\Delta} \Big)C's \geq s \geq |W_{\ell+1}| \geq |Y|,$$
which contradicts our assumption that $|U| < |Y| = |\cF_k|$. This shows that Hall's condition $|Y| \leq |U|$ holds, as desired. Therefore, there is a system of representatives for $\big(C(y)\big)_{y \in W_{\ell+1}}$. In other words, there is an injective function $\psi: W_{\ell+1} \rightarrow \cup_{y\in W_{\ell+1}}C(y)$ such that $\psi(y) \in C(y)$ for each $y \in W_{\ell+1}$.

We can now extend $\phi_{\ell}$ to get $\phi_{\ell+1}$ and define $C_{\ell+1}(z)$ for $z \in \cup_{j=\ell+2}^{Q}$. Firstly, let
$$
\phi_{\ell+1}(w)=\begin{cases}
			\phi_{\ell}(w), & \text{if } w \in \cup_{j=1}^{\ell} W_j,\\
            \psi(w), & \text{if } w \in W_{\ell+1}.
		 \end{cases}
$$
Since all pairs of vertices in $W_{\ell+1}$ are at distance at least $4$, each $z \in \cup_{j=\ell+2}^Q W_j$ has no more than one neighbour in $W_{\ell+1}$. Thus, for each such $z \in \cup_{j=\ell+2}^Q W_j$ we can set
$$ C_{\ell+1}(z)=\begin{cases}
			C_{\ell}(z), & \text{if } N_H(z) \cap W_{\ell+1}=\varnothing,\\
            C_{\ell}(z) \cap N_F(\phi_{\ell+1}(y)), & \text{if } N_H(z) \cap W_{\ell+1} = \{y\}.
		 \end{cases} $$
We now verify that $\phi_{\ell+1}$ and $C_{\ell+1}(z)$ for each $z\in \cup_{j=\ell+2}^Q W_{j}$ satisfy $(\cS_{\ell+1})$.

From property (a) of $(\cS_{\ell})$ together with $\phi_{\ell+1}(y) \in C(y) \subseteq C_{\ell}(y)$ for every $y \in W_{\ell+1}$ and the fact that $\psi$ is injective, it follows that $\phi_{\ell+1}$ is a partial embedding of $H[\bigcup_{j=1}^{\ell+1}W_j]$ into $F[\bigcup_{j=1}^{\ell+1}F_j]$.

We now check that parts (a) and (b) of $(\cS_{\ell+1})$ hold. Fix some $z \in \bigcup_{j=\ell+2}^{Q}W_j$. In case $N_H(z) \cap W_{\ell+1} = \varnothing$, we have $C_{\ell+1}(z) = C_{\ell}(z)$ and $\mathsf{ldeg}_g^{\ell+1}(z) = \mathsf{ldeg}_g^{\ell}(z)$, so (a) and (b) of $(\cS_{\ell+1})$ are satisfied for that $z$. Now suppose $N_H(z) \cap W_{\ell+1} = \{y\}$ (recall that this is the only other possible case). Since we defined $C_{\ell+1}(z) = C_{\ell}(z) \cap N_F(\phi_{\ell+1}(y))$, property (a) of $(\cS_{\ell+1})$ holds for $z$ in this case. Furthermore, by our choice $\phi_{\ell+1}(y) \in C(y)$ and by (b'), it follows that property (b) of $(\cS_{\ell+1})$ is satisfied.

Lastly, we check property (c) of $(\cS_{\ell+1})$. Fix an edge $zz'$ of $H$ such that $z,z' \in \bigcup_{j=\ell+2}^{Q}W_j$. There are three cases to consider, based on the size of $N_H(z) \cap W_{\ell+1}$ and $N_H(z') \cap W_{\ell+1}$ (recall that each of them has size either $0$ or $1$). Suppose first $N_H(z) \cap W_{\ell+1} = \varnothing$ and $N_H(z') \cap W_{\ell+1} = \varnothing$. Then we have $\mathsf{ldeg}^{\ell}_g(z) = \mathsf{ldeg}^{\ell+1}_g(z)$ and $\mathsf{ldeg}^{\ell}_g(z') = \mathsf{ldeg}^{\ell+1}_g(z')$, as well as $C_{\ell}(z) = C_{\ell+1}(z)$ and $C_{\ell}(z') = C_{\ell}(z')$, so part (c) of $(\cS_{\ell+1})$ follows directly from part (c) of $(\cS_{\ell})$. Next, suppose $N_H(z) \cap W_{\ell+1} =\{y\}$ and $N_H(z') \cap W_{\ell+1} = \varnothing$. Then property (c') together with the definition of $C_{\ell+1}(z)$ and $C_{\ell+1}(z')$ imply part (c) of $(\cS_{\ell+1})$. Finally, suppose $N_H(z) \cap W_{\ell+1} = \{y\}$ and $N_H(z') \cap W_{\ell+1} = \{y'\}$. Then it must be the case that $y=y'$, since otherwise $y$ and $y'$ would be two distinct vertices in $W_{\ell+1}$ connected by a path of length $3$ in $H$, namely the path $yzz'y'$, and this is impossible by our choice of $W_1, \ldots, W_Q$. Then in this case also, part (c) of $(\cS_{\ell+1})$ follows from (c') and our choice of $C_{\ell+1}(z)$ and $C_{\ell+1}(z')$.

We have shown that (a), (b), and (c) of $(\cS_{\ell+1})$ hold. This finishes the induction step and the proof of Lemma~\ref{lem:KRSS_generalization}.
\end{proof}

\bibliographystyle{plain}
\bibliography{bibliography.bib}
\end{document}